\documentclass [11pt,a4paper]{article}
\usepackage{amsmath}
\usepackage{amsthm}
\usepackage{amssymb}
\usepackage{amscd}
\usepackage{amsfonts}
\usepackage{amsbsy}
\usepackage{epsfig,afterpage}
\usepackage[dvips]{psfrag}
\usepackage{graphicx}
\usepackage{indentfirst, latexsym, bm,amssymb}
\usepackage{bbding}

\advance\textwidth by +1.0in \advance\textheight by +1.0in
\advance\oddsidemargin by -0.5in \advance\evensidemargin by -1.0in
\advance\topmargin by -0.5in
\newtheorem {theorem} {Theorem}%[section]
\newtheorem {proposition} [theorem]{Proposition}

\newtheorem {lemma}  [theorem]{Lemma}

\newcommand{\X}{{\mathcal{\bf} X}}

\newcommand{\R}{{\mathbb R}}
\newcommand{\Z}{{\mathbb Z}}

\def\a{\alpha}
\def\b{\beta}

\def\g{\gamma}

\def\e{\varepsilon}
\def\p{\partial}

\title{\large\bf Nonuniform Dichotomy Spectrum and Normal Forms
for Nonautonomous Differential Systems}
\author{\normalsize\bf\sc Xiang Zhang\footnote{\small
The author is partially supported by NNSF of China grant 11271252,  RFDP of Higher Education of China grant 20110073110054, and FP7-PEOPLE-2012-IRSES-316338 of Europe.\newline \quad $\ddag$ Submitted to JFA in JAN 1, 2011, and Accepted by JFA in JUL 30, 2014}\\
\normalsize\it Department of Mathematics, and MOE--LSC, Shanghai Jiao Tong
               University,  \\ \normalsize\it Shanghai 200240,
               People's Republic of China.
      \\ \normalsize  E-mail: xzhang@sjtu.edu.cn\\}
\date{}

\begin{document}
\maketitle

\begin{abstract}
\noindent The aim of this paper is to study the normal forms of nonautonomous differential systems. For doing so, we
first investigate the nonuniform dichotomy spectrum of the
linear evolution operators that admit a nonuniform exponential
dichotomy, where the linear evolution operators are defined by
nonautonomous differential equations $\dot x=A(t)x$ in $\R^n$. Using the nonuniform dichotomy spectrum we obtain the normal forms
of the nonautonomous linear differential equations. Finally we establish the finite jet normal forms of
the nonlinear differential systems $\dot x=A(t)x+f(t,x)$ in $\R^n$, which is based
on the nonuniform dichotomy spectrum and the normal forms of the
nonautonomous linear systems.

\hskip0.00011mm


\hskip0.00011mm

\end{abstract}

%\maketitle

\bigskip

\section{Introduction and statement of the main results}\label{s1}

\setcounter{section}{1} \setcounter{equation}{0}

The normal form theory in dynamical systems is to simplify ordinary differential equations through the change of variables.
This theory can be traced back to Poincar\'e \cite{Po1}. Some classical results in this direction for autonomous differential systems are the
Poincare--Dulac normal form theorem \cite{Po}, the Siegel's theorem
\cite{Si}, the Hartman--Grobman's theorem \cite{Gr,Ha}, the
Sternberg's theorem \cite{St,St58}, the Chen's theorem \cite{Ch}, the Takens' theorem \cite{Ta}
and so on. See also \cite{Bi,CYZ,IY,Li,Zh} and the recent survey paper \cite{Sto} and the references therein. For nonautonomous systems, Barreira and Valls had
established several results on the topological conjugacy between nonuniformly hyperbolic dynamical systems (see e.g.
\cite{BV06JDE}, \cite{BV07JFA}--\cite{BV08ETDS}). Using the resonance of the dichotomy
spectrum to study the normal forms of nonautonomous system,
Siegmund \cite{Si02JDE} obtained a finite order normal form, and Wu
and Li \cite{WL08} got analytic normal forms of a class of analytic nonautonomous differential systems. As our knowledge,
these last two papers are the only ones in which the
normal forms of nonautonomous systems via the dichotomy spectrums were studied.
Recently Li, Llibre and Wu \cite{LLW} and \cite{LW1} also had studied the normal forms of almost periodic differential and difference equations, respectively.
For random differential systems there also appeared some
results on normal forms \cite{LL1,LL2,LL3}, in which they extended the
Poincar\'e's, the Sternberg's and the Siegel's normal form theorems for
autonomous differential systems to random dynamical systems.

As well--known, the normal form theory has played important roles in the study of bifurcation and some related topics of dynamical systems. Recently this theory has been successfully applied to study the embedding flow problem of diffeomorphims, see for instance \cite{LLZ02,Zh1,Zh2,Zh3}.

In this paper we will study the normal forms of nonautonomous differential systems
with their linear parts admitting a nonuniform exponential dichotomy.
For this aim we first consider the nonautonomous linear differential systems in $\R^n$
\begin{equation}\label{1.1}
\dot x=A(t)x,
\end{equation}
with $A(t)\in M_n(\R)$ the set of square matrix functions of $n$th
order defined in $\R$, we assume in this paper that each
solution of system \eqref{1.1} is defined on the whole $\R$.
Denote by $\Phi(t,s)$ the evolution operator associated to system
\eqref{1.1}. Then we have
\[
x(t)=\Phi(t,s)x(s),\quad \Phi(t,s)\Phi(s,\tau)=\Phi(t,\tau) \mbox{
for all } t,s,\tau\in\R,
\]
where $x(t)$ is a solution of system \eqref{1.1}.

We say that system \eqref{1.1} admits a {\it nonuniform exponential dichotomy}
if there exists an invariant projection $P(t)\in M_n(\R)$ (where {\it invariant} means that $P(t)\Phi(t,s)=\Phi(t,s)P(s)$ for all $t,s\in\R$), and $K\ge
1$, $\alpha<0<\beta$ and $\mu,\nu\ge 0$ with $\alpha+\mu<0$,
$\beta-\nu>0$ and $\max\{\mu,\nu\}\le \min\{-\alpha,\beta\}$ such that
\begin{eqnarray*}
\|\Phi(t,s)P(s)\|&\le& Ke^{\a(t-s)+\mu|s|} \mbox{ for } t\ge s,\\
\|\Phi(t,s)(I-P(s))\|&\le& Ke^{\b(t-s)+\nu|s|} \mbox{ for } t\le
s.
\end{eqnarray*}
If $\mu=\nu=0$ it defines the uniform exponential dichotomy or
simply exponential dichotomy (see e.g. \cite{Si02JDDE,Co78}).
Barreira and Valls \cite{BV07JDDE} showed that the system in
$\R^2$
\[
\dot x=-(\omega+at\sin t)x,\quad \dot y=(\omega+at\sin t)y,
\]
admits a nonuniform exponential dichotomy but does not admit a
uniform exponential dichotomy.

In our definition of the nonuniform exponential dichotomy there appear the
extra conditions $\alpha+\mu<0$, $\beta-\nu>0$ and $\max\{\mu,\nu\}\le \min\{-\alpha,\beta\}$, which did not
appear explicitly in the definition of  \cite{BV06JDE,BV07JFA,BV08ETDS}. In fact, in their results on the conjugacy between two
nonautonomous dynamical systems they always assume that the nonuniform
constants $\mu$ and $\nu$ are sufficiently small, and consequently the extra conditions hold implicitly.

The {\it nonuniform dichotomy spectrum} of system \eqref{1.1} is the set
\[
\Sigma(A)=\{\g\in\R;\, \dot x=(A(t)-\gamma I)x \mbox{ admits no
nonuniform exponential dichotomy}\}.
\]
Its complement $\rho(A)=\R\setminus \Sigma(A)$ is called the {\it
resolvent set} of system \eqref{1.1}.

A {\it linear integral manifold} of system \eqref{1.1} is a nonempty set
$W$ of $\R\times\R^n$ satisfying $\{(t,\Phi(t,\tau)\xi);\,$ $t\in\R\}\subset W$  for each
$(\tau,\xi)\in W$, and for any given $\tau\in\R$ the fiber
$W(\tau)=\{\xi\in\R^n;\,(\tau, \xi)\in W\}$ is a linear subspace
of $\R^n$. In the following we also call $W$ {\it invariant} by \eqref{1.1}. We note that all the fibers $W(\tau)$ have the same
dimension, denoted by $\dim W$, and they form a vector bundle over
$\R$. A linear integral manifold is a topological manifold in
$\R\times\R^n$.

Let $W_1$ and $W_2$ be two linear integral manifolds of
\eqref{1.1}. Their {\it intersection} and {\it sum} are defined
respectively as
\begin{eqnarray*}
W_1\cap W_2&=&\{(\tau,\xi)\in\R\times\R^n;\,\xi\in W_1(\tau)\cap
W_2(\tau)\},\\
W_1+W_2&=&\{(\tau,\xi)\in\R\times\R^n;\,\xi\in
W_1(\tau)+W_2(\tau)\}.
\end{eqnarray*}
They are also linear integral manifolds. A sum of linear integral
manifolds $W_1,\ldots,W_k$ is called {\it Whitney--sum}, denoted
by $W_1\oplus\ldots\oplus W_k$, if $W_i\cap W_j=\R\times \{0\}$
for $1\le i\ne j\le k$.

For a $\gamma\in\R$ we define two subsets of $\R\times\R^n$:
\begin{equation}\label{1.2}
\begin{array}{l}
\mathcal
U_\gamma=\left\{(\tau,\xi)\in\R\times\R^n;\,\sup\limits_{t\ge
0}\|\Phi(t,\tau)\xi\|e^{-\gamma t}<\infty\right\},\\
\mathcal
V_\gamma=\left\{(\tau,\xi)\in\R\times\R^n;\,\sup\limits_{t\le
0}\|\Phi(t,\tau)\xi\|e^{-\gamma t}<\infty\right\}.
\end{array}
\end{equation}
In this paper the notations $\mathcal U_\gamma$ and $\mathcal
V_\gamma$ always denote the sets defined in \eqref{1.2},
respectively.

Our first result is on the structure of the nonuniform dichotomy
spectrum of system \eqref{1.1}. It is the generalization of the
spectral theorem of \cite{Si02JDDE} for the dichotomy spectrum to the
nonuniform dichotomy spectrum of system \eqref{1.1}.

\begin{theorem}\label{th1.1} For system \eqref{1.1}, the following statements hold.

\begin{itemize}
\item[$(a)$] The nonuniform dichotomy spectrum $\Sigma(A)$ of
system \eqref{1.1} is the union of $m$ disjoint closed intervals
in $\R$ $($called spectral intervals$)$ with $0\le m\le n$. Precisely, if $m=0$ then $\Sigma(A)=\emptyset$; if
$m=1$ then $\Sigma(A)=\mathbb R$ or $(-\infty,b_1]$ or $[a_1,b_1]$ or $[a_1,\infty)$; if $m>1$ then
$\Sigma(A)=I_1\cup [a_2,b_2]\cup\ldots\cup [a_{m-1},b_{m-1}]\cup I_m$ with $I_1=[a_1,b_1]$ or $(-\infty,b_1]$ and $I_m=[a_m,b_m]$ or $[a_m,\infty)$,
where $a_i\le b_i<a_{i+1}$ for $i=1,\ldots,m-1$.

\item[$(b)$]  If $m\ge 1$ and $\Sigma(A)\ne\R$, assume that
\[
\Sigma(A)=I_1\cup[a_2,b_2]\cup\ldots\cup
[a_{m-1},b_{m-1}]\cup I_m,
\]
with $I_1,I_m$ and $a_i,b_i$ given in statement $(a)$. If $I_1=[a_1,b_1]$ and $I_m=[a_m,b_m]$, set $b_0=-\infty$ and $a_{m+1}=\infty$, and choose
$\gamma_i\in(b_i,a_{i+1})$ for $i=0,1,\ldots,m$, we have the
linear integral manifolds $\mathcal U_{\gamma_i}$ and $\mathcal
V_{\gamma_i}$ for $i=0,1,\ldots,m$.  If $I_1=(-\infty,b_1]$, we choose $\gamma_0<b_1$ and
set $\mathcal U_{\gamma_0}=\R\times\{0\}$ and $\mathcal
V_{\gamma_0}=\R\times\R^n$.  If $I_m=[a_m,\infty)$, we choose
$\gamma_m>a_m$ and set $\mathcal U_{\gamma_{m}}=\R\times\R^n$ and
$\mathcal V_{\gamma_{m}}=\R\times\{0\}$. Define
\[
\mathcal W_0=\mathcal U_{\gamma_0},\quad \mathcal W_i=\mathcal
U_{\gamma_i}\cap\mathcal V_{\gamma_{i-1}} \,\mbox{ for }
i=1,\ldots,m, \quad \mathcal W_{m+1}=\mathcal V_{\gamma_m}.
\]
Then $\dim\mathcal W_i\ge 1$ for $i=1,\dots,m$ and
\[
\R\times\R^n=\mathcal W_0\oplus\mathcal
W_1\oplus\ldots\oplus\mathcal W_{m+1}.
\]
\end{itemize}
\end{theorem}

The linear integral manifold $\mathcal W_i$ is called a {\it
spectral manifold} for $i=0,\ldots,m+1$. We shall see from
Proposition \ref{po2.3} below that the spectral manifold $\mathcal
W_i$ is independent of the choice of $\gamma_i$.

Next we present a sufficient condition for a nonuniform
dichotomy spectrum to be nonempty and bounded.

The evolution operator $\Phi(t,s)$ of $\dot x=A(t)x$ has a {\it
nonuniformly bounded growth} if there exist $K\ge 1$, $a\ge 0$ and
$\e\ge 0$ such that
\begin{equation}\label{1.3}
\|\Phi(t,s)\|\le Ke^{a|t-s|+\e|s|},\quad t,s\in\R.
\end{equation}
If $\e=0$ the evolution operator has the so--called bounded growth
(see \cite{Si02JDDE})

\begin{theorem}\label{th1.2}
Assume that the evolution operator of system \eqref{1.1} has a nonuniformly
bounded growth. The following statements hold.
\begin{itemize}
\item[$(a)$] The nonuniform dichotomy spectrum $\Sigma(A)$ of
system \eqref{1.1} is nonempty and bounded, i.e.,
$\Sigma(A)=[a_1,b_1]\cup\ldots\cup [a_m,b_m]$ with $m\ge 1$ and
$-\infty<a_1\le b_1<\ldots<a_m\le b_m<\infty$.

\item[$(b)$] $\mathcal W_0\oplus\mathcal W_1\oplus\ldots\oplus \mathcal W_m\oplus\mathcal W_{m+1}=\R\times
\R^n$,  where $\mathcal W_i$'$s$ are the spectral manifolds defined in Theorem
\ref{th1.1}.
\end{itemize}
\end{theorem}

For autonomous linear systems in $\R^n$ it is well known that they
can be transformed into normal forms with their coefficient matrices in the Jordan type through some nondegenerate linear
changes of variables. Using the dichotomy spectrum Siegmund
\cite{Si02JLMS} provided a method to study the normal forms of
nonautonomous linear systems. Here we extend his method to study
the normal form of nonautonomous linear differential system using
the nonuniform dichotomy spectrum.

As first defined in \cite{Sie2014}, we say that system \eqref{1.1} and the system
\begin{equation}\label{1.4}
\dot y=B(t)y,
\end{equation}
are {\it nonuniformly kinematically similar} if there exists a
differentiable matrix function $S:\R\rightarrow {\rm GL}_n(\R)$
satisfying
\begin{equation}\label{enks}
  \|S(t)\|\le M_\varepsilon e^{\varepsilon |t|},\quad
 \|S(t)^{-1}\|\le M_\varepsilon e^{\varepsilon |t|},\qquad \mbox{ for all } t\in\mathbb R,
\end{equation}
with $M_\varepsilon>0$ a constant, such that $x(t)=S(t)y(t)$ transforms \eqref{1.1} into \eqref{1.4}. Correspondingly, the $S(t)$ satisfying \eqref{enks} is called a {\it nonuniform Lyapunov matrix}, and the change of variables $x(t)=S(t)y(t)$ is a {\it nonuniform Lyapunov transformation}.

The following result characterizes the normal forms of
nonautonomous linear differential systems via their nonuniform
dichotomy spectrums.

\begin{theorem}\label{th1.3}
Assume that $A(t)$ is differentiable, and that the evolution
operator of system \eqref{1.1} has a nonuniformly bounded growth.
Let $\Sigma(A)=[a_1,b_1]\cup \ldots \cup [a_m,b_m]$ with $-\infty<a_1\le
b_1< \ldots < a_m\le b_m<\infty $ be the nonuniform dichotomy spectrum. Then system \eqref{1.1} is
nonuniformly kinematically similar to
\begin{equation}\label{1.5}
\dot y=\left(\begin{array}{ccccc}
B_0(t) &          &        &              &\\
       & B_1(t)   &        &              &\\
       &          & \ddots &             &\\
       &          &        &   B_{m}(t)  &  \\
      &           &       &            & B_{m+1}(t)
 \end{array}
 \right)y,
\end{equation}
where $B_i(t):\,\R\rightarrow \R^{n_i\times n_i}$ are
differentiable with $n_i={\rm dim}\,\mathcal W_i$, $\Sigma(B_0)=\Sigma(B_{m+1})=\emptyset$ and
$\Sigma(B_i)=[a_i,b_i]$ for $i=1,\dots,m$. Recall that $\mathcal W_0,\mathcal W_1,\ldots,\mathcal W_{m},\mathcal W_{m+1}$ are the
corresponding spectral manifolds.
\end{theorem}

Now we use the nonuniform dichotomy spectrums and the normal forms
for nonautonomous linear differential systems to study the normal forms of  nonautonomous
nonlinear differential systems.

Consider the  nonautonomous nonlinear differential system
\begin{equation}\label{1.6}
\dot x=A(t)x+f(t,x),\quad x\in (\mathbb R^n,0),
\end{equation}
where $f(t,x)=O(|x|^2)$ is an analytic function.

Assume that the evolution operator of the linear system $\dot
x=A(t)x$ has a nonuniformly bounded growth. Then its nonuniform
dichotomy spectrum is $\Sigma(A)=[a_1,b_1]\cup\ldots\cup
[a_m,b_m]$ with $m\ge 1$ and $-\infty<a_1\le b_1<\ldots<a_m\le
b_m<\infty$. Hence by Theorem \ref{th1.3}  system \eqref{1.6} is
equivalent to
\begin{equation}\label{1.7}
\dot x_i=A_i(t) x_i+f_i(t,x),\quad i=1,\ldots,m,
\end{equation}
where $A_i$ is an $n_i\times n_i$ matrix with $n_1+\ldots+n_m=n$
and $\Sigma(A_i)=[a_i,b_i]$ for $i=1,\ldots,m$.
So there exist $\varepsilon_i>0$ suitably small such that for $\sigma_i\in [a_i-\varepsilon_i,a_i)$ and $\rho_i\in (b_i,b_i+\varepsilon_i]$ systems $\dot z=(A_i(t)-\sigma_i I)z$ and $\dot
z=(A_i(t)-\rho_i I)z$ admit nonuniform exponential dichotomies.
Hence there exist $K_i\ge 1$, $\alpha_i<0$, $\beta_i>0$, and
$\mu_i,\nu_i\ge 0$ with $\alpha_i+\mu_i<0$ and $\beta_i-\nu_i>0$
such that
\begin{equation}\label{4.2100}
\begin{array}{l}
\|\Phi_{A_i}(t,s)\|\le
K_ie^{\rho_i(t-s)}e^{\alpha_i(t-s)+\mu_i|s|}\quad \mbox{\rm for } t\ge s,\\
\|\Phi_{A_i}(t,s)\|\le
K_ie^{\sigma_i(t-s)}e^{\beta_i(t-s)+\nu_i|s|}\quad \mbox{\rm for }
t\le s.
\end{array}
\end{equation}

In what follows we study only system \eqref{1.7}.
Expanding $f_i(t,x)$ in the Taylor series
\[
f_i(t,x)\sim\sum\limits_{|l|=2}\limits^{\infty}f_{il}(t)x^l, \quad
i=1,\ldots,m,
\]
where $l=(l_1,\ldots,l_n)\in\Z_+^n$ are  multiple indices with
$\mathbb Z_+=\mathbb N\cup\{0\}$, $x^l=x_1^{l_1}\ldots x_n^{l_n}$
and $|l|=l_1+\ldots+l_n$.

In \eqref{1.7} a monomial, say $f_{il}(t)x^l$, is {\it nonresonant} if
\[
[a_i,b_i]\cap\sum\limits_{j=1}\limits^m\tau_j[a_j,b_j]=\emptyset,
\]
where the sum and the multiplication of intervals are defined as
\[
[a,b]+[c,d]=[a+c,b+d],\qquad  k [a,b]=[k a,k b],
\]
and $\tau=(\tau_1,\ldots,\tau_m)$ is the image of
$l=(l_1,\ldots,l_n)\in\Z_+^n$ under the mapping
\begin{eqnarray*}
\mathcal N:\R^n & \longrightarrow & \R^m\\
l\,\,\,  & \longrightarrow &
\tau=(l_1+\ldots+l_{n_1},l_{n_1+1}+\ldots+l_{n_1+n_2},\ldots,l_{n-n_{m}+1}+\ldots+l_{n}).
\end{eqnarray*}

The notion nonresonance for nonautonomous differential systems is
an extension of the one for autonomous system $\dot x=Ax+f(x)$,
where the nonresonant condition is
\[
\lambda_i\ne \sum\limits_{j=1}\limits^nk_j\lambda_j,\quad
k=(k_1,\ldots,k_n)\in\Z_+^n, \,\, |k|\ge 2,
\]
with $\lambda=(\lambda_1,\ldots,\lambda_n)$ the eigenvalues of the
constant matrix $A$.

We say that system \eqref{1.7} is in the {\it normal form} if its
nonlinear terms are all resonant. The transformation sending
\eqref{1.7} to its normal form is called a {\it normalization}.
Usually the normalization is not unique. If the normalization
contains only nonresonant terms, then it is called a {\it
distinguished normalization}. The corresponding normal form system
is called in the {\it distinguished normal form}. We note that for a
given differential system the Taylor expansion of its distinguished normalization is
unique. Of course, if the distinguished normalization is analytic, then itself is unique.

Let $f(t,x)=(f_1(t,x),\ldots,f_m(t,x))$ have the Taylor expansion
$ f_j(t,x)\sim\sum\limits_{s=2}\limits^\infty \widetilde f_{js}(t,x)$,
where $\widetilde f_{js}$ is a vector--valued homogeneous polynomial
of degree $s$ in $x$ with its coefficients being the functions of $t$.

\begin{theorem}\label{th1.4}
Assume that system \eqref{1.7} is analytic or $C^\infty$ and that the evolution
operator of the linear system associated with \eqref{1.7} has a
nonuniformly bounded growth. Let $\Sigma(A)=[a_1,b_1]\cup \ldots
\cup [a_m,b_m]$ with $-\infty<a_1\le b_1< \ldots < a_m\le b_m<\infty$ be the
nonuniform dichotomy spectrum, and let $\alpha_i,\beta_i,\mu_i $ and $\nu_i$ be the data defined in \eqref{4.2100}. Set $\varrho=\max\{\mu_i,\nu_i;\, i=1,\ldots,m\}$, and $\sigma=\min\{-\alpha_j,\beta_j,\,j=1,\ldots,m\}$. If $\sigma/\varrho>4$ and there exists a positive number $k\in(3,\sigma/\varrho)$
such that the coefficient vectors of
$\widetilde f_s=(\widetilde f_{1s},\ldots,\widetilde f_{ms})$ according to the base $\{x^{\tau}e_j:\,\tau\in\mathbb Z_+^n,|\tau|=s,j=1,\ldots,n\}$, denoted by $p_s(t)$, satisfy
\begin{equation}\label{eth141}
\|p_s(t)\|\le d_s\exp\left(-\left((s-1)k-\frac{(s+3)(s-2)}{2}\right)\varrho |t|\right),\quad\mbox{ for } 2\le s<2k-4,
\end{equation}
then there exists a near identity polynomial map of degree $2k-5$ under which system \eqref{1.7} is transformed into
\begin{equation}\label{1.71}
\dot y=Ay+g(t,y)+h(t,y),
\end{equation}
where $g(t,y)$ consists of the resonant homogeneous polynomials in $y$ of degrees from $2$ to $2k-5$ with coefficients being the functions of $t$, which are uniformly convergent to zero when $|t|\rightarrow \infty$, and $h(t,y)=O(|y|^{2k-4})$.
\end{theorem}

In the last theorem we have several restricted conditions. We should say that except the one on the modulus $\|p_s(t)\|$, the others are natural.
For instance $\sigma/\varrho>4$ holds provided that the nonuniform exponents $\mu_i,\nu_i$ are sufficiently small. The condition on the modulus $\|p_s(t)\|$ is also natural in some sense, because if $\|p_s(t)\|$ increases
too fast as $|t|$ increases, any orbit starting in a small neighborhood of the origin will rapidly leave the neighborhood,  and so the theorem will not be correct. If $\varrho=0$ we are in the case of the uniform dichotomy spectrum.

We mention that if an analytic or a $C^\infty$ system \eqref{1.7} has its linear part satisfying a nonuniform exponential dichotomy, it is nearly impossible to get an analytic or a $C^\infty$ normalization which transforms system \eqref{1.7} to its normal form (of course, if system \eqref{1.7} is a polynomial one, the normalization may exist). Also Theorem \ref{th1.4} holds for $C^{2k-4}$ differential systems. These can be seen from the proof of Theorem \ref{th1.4}.

We also mention that even for a $C^k$ $(2<k<\infty)$ smooth autonomous differential system of form \eqref{1.7}, if $n>2$ there is no satisfactory results on the regularity of the normalization which transforms system \eqref{1.7} to a polynomial normal form. For $n=2$ this problem was solved by Stowe \cite{Stjde}.

This paper is organized as follows. In the next section we shall prove Theorems \ref{th1.1} and \ref{th1.2}. The proof of Theorems \ref{th1.3} and \ref{th1.4} will be given in Sections \ref{s3} and \ref{s4}, respectively.

\section{Proof of Theorems \ref{th1.1} and \ref{th1.2}}\label{s2}

\setcounter{section}{2}
\setcounter{equation}{0}\setcounter{theorem}{0}

For proving Theorems \ref{th1.1} and \ref{th1.2} we need some
basic results which characterize the nonuniform dichotomy
spectrum. The ideas of the proofs partially follow  from
\cite{Si02JDDE}.

\subsection{The basic results}

This subsection is a preparation for proving Theorems \ref{th1.1}
and \ref{th1.2}.

\begin{proposition}\label{po2.1}
Let $\mathcal U_\gamma,\,\mathcal V_\gamma$ be the subsets of
$\R\times\R^n$ defined in \eqref{1.2}. The following statements
hold.
\begin{itemize}
\item[$(i)$] $\mathcal U_\gamma$ and $\mathcal V_\gamma$ are
linear integral manifolds of system \eqref{1.1}

\item[$(ii)$] If $\gamma_1\le \gamma_2$ then $\mathcal
U_{\gamma_1}\subseteq \mathcal U_{\gamma_2}$ and $\mathcal
V_{\gamma_1}\supseteq \mathcal V_{\gamma_2}$.
\end{itemize}
\end{proposition}

\begin{proof} $(i)$ For any $(\tau,\xi)\in\mathcal U_\gamma$, by
definition we only need to prove $(s,\Phi(s,\tau)\xi)\in\mathcal
U_\gamma$ for all $s\in\R$. In fact, it follows from the fact that
\[
\sup\limits_{t\ge 0}\|\Phi(t,s)\Phi(s,\tau)\xi\|e^{-\gamma t}=
\sup\limits_{t\ge 0}\|\Phi(t,\tau)\xi\|e^{-\gamma t}<\infty.
\]
The proof for $\mathcal V_{\gamma}$ follows from the same
arguments as those for $\mathcal U_{\gamma}$.

\noindent $(ii)$ The claim $\mathcal U_{\gamma_1}\subseteq\mathcal
U_{\gamma_2}$ follows easily from $-\gamma_1 t\ge -\gamma_2 t$ for
$t\ge 0$. A similar argument works with $\mathcal V_{\gamma}$.
\end{proof}

\begin{proposition}\label{po2.2}
For $\gamma\in\R$, if
\begin{equation}\label{2.1}
\dot x=(A(t)-\gamma I)x,
\end{equation}
admits a nonuniform exponential dichotomy with an invariant
projection $P$, then we have
\[
\mathcal U_\gamma=\mbox{\rm Im} P,\quad \mathcal
V_\gamma=\mbox{\rm Ker} P\,\,\,\, \mbox{ and } \,\,\,\, \mathcal
U_\gamma\oplus \mathcal V_\gamma=\R\times \R^n,
\]
where $\mbox{\rm Im}P$ and ${\rm Ker}P$ denote the image and
kernel of the projection $P$, respectively.
\end{proposition}

\begin{proof}
Let $\Phi(t,s)$ be the evolution operator of $\dot x=A(t)x$. Some
easy calculations show that
$\Phi_\gamma(t,s)=e^{-\gamma(t-s)}\Phi(t,s)$ is an evolution
operator of \eqref{2.1}, and that $P(t)$ is an invariant
projection of $\Phi(t,\tau)$ if and only if it is an invariant
projection of $\Phi_\gamma(t,\tau)$. By the assumption there exist
$K_\gamma\ge 1$, $\alpha_\gamma<0,\beta_\gamma>0$ and
$\mu_\gamma,\nu_\gamma\ge 0$ with $\alpha+\mu_\gamma<0$ and
$\beta-\nu_\gamma>0$ such that
\begin{eqnarray*}
\|\Phi_\gamma(t,s)P(s)\|&\le &K_\gamma
e^{\alpha_\gamma(t-s)+\mu_\gamma|s|} \mbox{
for all } t\ge s,\\
\|\Phi_\gamma(t,s)(I-P(s))\|&\le &K_\gamma
e^{\beta_\gamma(t-s)+\nu_\gamma|s|} \mbox{ for all } t\le s.
\end{eqnarray*}

First we prove $\mathcal U_\gamma\subset{\rm Im} P$.  For any
$(\tau,\xi)\in\mathcal U_\gamma$, by definition there exists a
constant $c_\gamma$ such that
\[
\|\Phi(t,\tau)\xi\|\le c_\gamma e^{\gamma t} \mbox{ for all } t\ge
0.
\]
It follows that
\[
\|\Phi_\gamma(t,\tau)\xi\|=e^{-\gamma(t-\tau)}\|\Phi(t,\tau)\xi\|\le
c_\gamma e^{\gamma\tau} \mbox{ for all } t\ge 0.
\]
Set $\xi=\xi_1+\xi_2$ with $\xi_1\in{\rm Im}P(\tau)$ and
$\xi_2\in{\rm Ker}P(\tau)$. Since
$P(t)\Phi_\gamma(t,\tau)=\Phi_\gamma(t,\tau)P(\tau)$, we have
\[
\xi_2=(I-P(\tau))\xi=\Phi_\gamma(\tau,t)(I-P(t))\Phi_\gamma(t,\tau)\xi.
\]
These yield that for $t\ge\max\{0,\tau\}$
\[
\|\xi_2\|=K_\gamma
e^{\beta_\gamma(\tau-t)+\nu_\gamma|t|}\|\Phi_\gamma(t,\tau)\xi\|\le
K_\gamma c_\gamma
e^{-(\beta_\gamma-\nu_\gamma)t+(\beta_\gamma+\gamma)\tau}.
\]
Hence we have $\xi_2=0$ because $\beta_\gamma-\nu_\gamma>0$, and
consequently $\xi=\xi_1\in{\rm Im}P(\tau)$. This proves that
$\mathcal U_\gamma\subset{\rm Im}P$.

For proving ${\rm Im} P\subset \mathcal U_\gamma$, we assume that
$\tau\in\R$, $\xi\in{\rm Im}P(\tau)$. Then $P(\tau)\xi=\xi$. For
$t\ge\max\{0,\tau\}$ we have
\[
\|\Phi(t,\tau)\xi\|e^{-\gamma
t}=e^{-\gamma\tau}\|\Phi_\gamma(t,\tau)P(\tau)\xi\|\le K_\gamma
e^{\alpha_\gamma t-(\gamma+\alpha_\gamma)\tau+\mu_\gamma|\tau|}.
\]
This implies that $(\tau,\xi)\in \mathcal U_\gamma$ because
$\alpha_\gamma<0$, and so ${\rm Im}P\subset \mathcal U_\gamma$.
This proves that ${\rm Im}P= \mathcal U_\gamma$.

Similarly using the assumption $\alpha_\gamma+\mu_\gamma<0$ we can
prove that $\mathcal V_\gamma={\rm Ker}P$. Finally the equality
$\mathcal U_\gamma\oplus\mathcal V_\gamma=\R\times\R^n$ follows
from $\mathcal U_\gamma={\rm Im}P$ and $\mathcal V_\gamma={\rm
Ker}P$.
\end{proof}

The next results characterize the resolvent set and the linear
integral manifolds.

\begin{proposition}\label{po2.3}
The resolvent set $\rho(A)$ is open. If $\gamma\in\rho(A)$ and
$J\subset \rho(A)$ is an interval containing $\gamma$, then
\[
\mathcal U_\eta=\mathcal U_\gamma,\quad \mathcal V_\eta=\mathcal
V_\gamma \quad \mbox{ for all } \eta\in J.
\]
\end{proposition}

\begin{proof}
For $\gamma\in\rho(A)$, by definition $\dot x=(A(t)-\gamma I)x$
admits a nonuniform exponential dichotomy with an invariant
projection $P(t)$. So there exist $K\ge 1$, $\alpha<0$, $\beta>0$
and $\mu,\nu\ge 0$ with $\alpha+\mu<0$ and $\beta-\nu>0$ such that
\begin{eqnarray*}
\|\Phi_\gamma(t,s)P(s)\|&\le& K e^{\alpha(t-s)+\mu|s|} \mbox{ for } t\ge s,\\
\|\Phi_\gamma(t,s)(I-P(s))\|&\le& K e^{\beta(t-s)+\nu|s|} \mbox{
for } t\le s.
\end{eqnarray*}
Set $0<\sigma<\min\{(\beta-\nu)/2,-(\alpha+\mu)/2\}$. For
$\eta\in(\gamma-\sigma,\gamma+\sigma)$, it is easy to see that
$P(t)$ is an invariant projection of the evolution operator
$\Phi_\eta(t,s)=e^{-\eta(t-s)}\Phi(t,s)$ of system $\dot
x=(A(t)-\eta I)x$. Moreover we have
\begin{eqnarray*}
\|\Phi_\eta(t,s)P(s)\|&=&e^{(\gamma-\eta)(t-s)}\|\Phi_\gamma(t,s)P(s)\|\le
Ke^{(\gamma-\eta+\alpha)(t-s)+\mu|s|} \,\,\mbox{ for } t\ge s,\\
\|\Phi_\eta(t,s)(I-P(s))\|&=&e^{(\gamma-\eta)(t-s)}\|\Phi_\gamma(t,s)(I-P(s))\|\le
Ke^{(\gamma-\eta+\beta)(t-s)+\nu|s|} \,\,\mbox{ for } t\le s.
\end{eqnarray*}
It follows from the choice of $\sigma$ and $\eta$ that
$\alpha^*=\gamma-\eta+\alpha\le \alpha^*+\mu<0$ and
$\beta^*=\gamma-\eta+\beta\ge \beta^*-\nu>0$. This proves that
$\dot x=(A(t)-\eta I)x$ admits a nonuniform exponential dichotomy
for all $\eta\in (\gamma-\sigma,\gamma+\sigma)$, and consequently
$(\gamma-\sigma,\gamma+\sigma)\subset \rho(A)$. This proves that
$\rho(A)$ is an open set.

For $\eta\in(\gamma-\sigma,\gamma+\sigma)$, the above proof shows
that systems $\dot x=(A(t)-\eta I)x$ and $\dot x=(A(t)-\gamma I)x$
both admit the nonuniform exponential dichotomy with the same
invariant projection $P(t)$. By Proposition \ref{po2.2} it holds
that $\mathcal U_\eta=\mathcal U_\gamma={\rm Im}P$ and $\mathcal
V_\eta=\mathcal V_\gamma={\rm Ker}P$.

For any given $\gamma^*\in J$, without loss of generality we
assume that $\gamma^*\le\gamma$. For each $\eta\in
[\gamma^*,\gamma]$ there exists an open set
$(\eta-\sigma_\eta,\eta+\sigma_\eta)\subset J$ with
$\sigma_\eta>0$ such that $\mathcal U_\zeta=\mathcal U_\eta$ and
$\mathcal V_\zeta=\mathcal V_\eta$ for
$\zeta\in(\eta-\sigma_\eta,\eta+\sigma_\eta)$. Since this kind of
intervals cover $[\gamma^*,\gamma]$, we get that  $\mathcal
U_\gamma=\mathcal U_{\gamma^*}$ and $\mathcal V_\gamma=\mathcal
V_{\gamma^*}$. By the arbitrariness of $\gamma^*\in J$ we can
finish the proof of the proposition.
\end{proof}

Let $\gamma_1,\gamma_2\in\rho(A)$. By Proposition \ref{po2.1}
$\mathcal U_{\gamma_2}$ and $\mathcal V_{\gamma_1}$ are both
linear integral manifolds. The following result characterizes
their intersection.

\begin{proposition}\label{po2.4}
For $\gamma_1,\gamma_2\in\rho(A)$ and $\gamma_1<\gamma_2$, set
$\mathcal W=\mathcal U_{\gamma_2}\cap\mathcal V_{\gamma_1}$. The
following conditions are equivalent.
\[
(a)\, \mathcal W\ne \R\times \{0\};\,\,\, (b)\,
[\gamma_1,\gamma_2]\cap\Sigma(A)\ne \emptyset;\,\,\, (c)\,
\dim\mathcal U_{\gamma_1}<\dim\mathcal U_{\gamma_2};\,\,\, (d)\,
\dim\mathcal V_{\gamma_1}>\dim\mathcal V_{\gamma_2}.
\]
\end{proposition}

\begin{proof}
The equivalence between $(c)$ and $(d)$ follows easily from
Proposition \ref{po2.2}.

\noindent {\it The condition $(c)$ implies $(a)$}. Since $(\mathcal
U_{\gamma_2}\cup\mathcal V_{\gamma_1})\setminus(\mathcal
U_{\gamma_2}\cap\mathcal V_{\gamma_1})\subset\R\times\R^n$ and
$\mathcal U_{\gamma_1}<\mathcal U_{\gamma_2}$, we have
\[
\dim\mathcal W=\dim(\mathcal U_{\gamma_2}\cap\mathcal
V_{\gamma_1})\ge\dim\mathcal U_{\gamma_2}+\dim\mathcal
V_{\gamma_1}-n>\dim\mathcal U_{\gamma_1}+\dim\mathcal
V_{\gamma_1}-n=0.
\]
So $\mathcal W\ne \R\times\{0\}$. This proves $(a)$.

\noindent {\it The condition $(a)$ implies $(b)$}. By contradiction we have
$[\gamma_1,\gamma_2]\subset\rho(A)$. So, it follows from Propositions
\ref{po2.3} and \ref{po2.2} that
\[
\mathcal U_{\gamma_2}\cap\mathcal V_{\gamma_1}= \mathcal
U_{\gamma_1}\cap\mathcal V_{\gamma_1}=\R\times\{0\}.
\]
This is in contradiction with $(a)$, and consequently $(b)$
follows.

\noindent {\it The condition $(b)$ implies $(c)$}. If not, since $\mathcal
U_{\gamma_1}\subseteq\mathcal U_{\gamma_2}$ by Proposition
\ref{po2.2} we have $\dim\mathcal U_{\gamma_1}=\dim\mathcal
U_{\gamma_2}$. It follows that $\dim\mathcal
U_{\gamma_1}(\tau)=\dim\mathcal U_{\gamma_2}(\tau)$ for all
$\tau\in\R$. But $\mathcal U_{\gamma_1}(\tau)$ and $\mathcal
U_{\gamma_2}(\tau)$ are linear subspaces of $\R^n$, we must have
$\mathcal U_{\gamma_1}(\tau)=\mathcal U_{\gamma_2}(\tau)$, and
consequently $\mathcal U_{\gamma_1}=\mathcal U_{\gamma_2}$. By the
equivalence of $(c)$ and $(d)$ we also have $\mathcal
V_{\gamma_1}=\mathcal V_{\gamma_2}$. This implies via Proposition
\ref{po2.2} that the nonuniform exponential dichotomies of $\dot
x=(A(t)-\gamma_1 I)x$ and $\dot x=(A(t)-\gamma_2 I)x$ involve the
same invariant projection $P(t)$. So there exist $K_i\ge 1$,
$\alpha_i<0$, $\beta_i>0$ and $\mu_i,\nu_i\ge 0$ with
$\alpha_i+\mu_i<0$ and $\beta_i-\nu_i>0$ for $i=1,2$ such that
\begin{eqnarray*}
\|\Phi_{\gamma_i}(t,s)P(s)\|&\le & K_ie^{\alpha_i(t-s)+\mu_i|s|} \,\mbox{ for } t\ge s,\\
\|\Phi_{\gamma_i}(t,s)(I-P(s))\|&\le& K_ie^{\beta_i(t-s)+\nu_i|s|}
\,\mbox{ for } t\le s.
\end{eqnarray*}
For $\gamma\in[\gamma_1,\gamma_2]$, take
$\alpha=\gamma_1-\gamma+\alpha_1$,
$\beta=\gamma_2-\gamma+\beta_2$, $\mu=\mu_1$, $\nu=\nu_2$ and
$K=\max\{K_1,K_2\}$, we have
\begin{eqnarray*}
\|\Phi_{\gamma}(t,s)P(s)\|&=&e^{(\gamma_1-\gamma)(t-s)}\|\Phi_{\gamma_1}(t,s)P(s)\|\le
K e^{\alpha(t-s)+\mu|s|} \,\mbox{ for } t\ge s,\\
\|\Phi_{\gamma}(t,s)(I-P(s))\|&=&e^{(\gamma_2-\gamma)(t-s)}\|\Phi_{\gamma_2}(t,s)(I-P(s))\|
\le K e^{\beta(t-s)+\nu|s|} \,\mbox{ for } t\le s.
\end{eqnarray*}
This proves that $\gamma\in\rho(A)$ and consequently
$[\gamma_1,\gamma_2]\subset\rho(A)$, a contradiction with the
assumption $(b)$. Hence $(c)$ holds. We complete the proof of the
proposition. \end{proof}

\subsection{Proof of Theorem \ref{th1.1}}

\noindent $(a)$ By Proposition \ref{po2.3} $\Sigma(A)$ is closed.
We now prove that the number of spectral intervals is no more
than $n$.

Since $\Sigma(A)\subset \R$ is closed, it is either empty or
consists of $m$ closed intervals with vanishing intersection. By
contradiction we assume that $m>n$. Set
$\Sigma(A)=[a_1,b_1]\cup\ldots\cup [a_n,b_n]\cup\ldots\cup
[a_m,b_m]$ with $a_1\le b_1<a_2\le b_2<\ldots<a_n\le
b_n<\ldots<a_m\le b_m$. Remark that we have the possibility with either $a_1=-\infty$, or
$b_m=\infty$, or both of them.  If it is the case, for instance $a_1=-\infty$ we take $[a_1,b_1]$ as $(-\infty,b_1]$. Choose $\gamma_i\in(b_i,a_{i+1})$ for
$i=1,\ldots,n$, we have the linear integral manifolds $\mathcal
U_{\gamma_i}$ and $\mathcal V_{\gamma_i}$ for $i=1,\dots,n$.

From Proposition \ref{po2.4} we get that
\[
\dim\mathcal U_{\gamma_1}<\dim\mathcal
U_{\gamma_2}<\ldots<\dim\mathcal U_{\gamma_n}\le n.
\]
So we must have either $\dim\mathcal U_{\gamma_1}=0$ or
$\dim\mathcal U_{\gamma_n}=n$.

If $\dim\mathcal U_{\gamma_1}=0$, i.e. $\mathcal
U_{\gamma_1}=\R\times\{0\}$, it follows from Proposition
\ref{po2.2} that $\mathcal V_{\gamma_1}=\R\times\R^n$, and the
invariant projection $P(t)=0$. By the definition of the nonuniform
exponential dichotomy we can prove easily that $\dot
x=(A(t)-\gamma I)x$ for all $\gamma<\gamma_1$ admits a nonuniform
exponential dichotomy with the invariant projection $P(t)$. This
verifies that $(-\infty,\gamma_1]\subset \rho(A)$. We are in
contradiction with the choice of $\gamma_1$.

If $\dim\mathcal U_{\gamma_n}=n$, i.e. $\mathcal
U_{\gamma_n}=\R\times\R^n$, Proposition \ref{po2.2} shows that the
invariant projection is $P(t)=I$. Then working in a similar way
to the proof of the case $\dim\mathcal U_{\gamma_1}=0$, we can prove that
$(\gamma_n,\infty)\subset \rho(A)$, a contradiction with the
choice of $\gamma_n$. Hence we must have $m\le n$. This proves
statement $(a)$.

\noindent $(b)$ First we claim that $\dim\mathcal W_i\ge 1$ for
$i=1,\ldots,m$.

We now prove this claim. For $i=1$, if $a_1\ne -\infty$ then
$\gamma_0,\gamma_1\in\rho(A)$ and
$[\gamma_0,\gamma_1]\cap\Sigma(A)\ne\emptyset$. Proposition
\ref{po2.4} shows that $\mathcal U_{\gamma_1}\supsetneq\mathcal
U_{\gamma_0}$. Since $\mathcal U_{\gamma_0}\oplus\mathcal
V_{\gamma_0}=\R\times\R^n$, we must have $\mathcal W_1=\mathcal
U_{\gamma_1}\cap\mathcal V_{\gamma_0}\supsetneq \mathcal
U_{\gamma_0}\cap\mathcal V_{\gamma_0}$. This implies that
$\dim\mathcal W_1\ge 1$ because $\mathcal W_1$ is a linear
integral manifold.

If $a_1= -\infty$ then $\mathcal W_1=\mathcal U_{\gamma_1}$
because $\mathcal V_{\gamma_0}=\R\times\R^n$.  By contradiction we
assume that $\dim\mathcal W_1=0$, i.e. $\mathcal W_1=\R\times
\{0\}$. Then $P(t)=0$ is the invariant projection associated with the
nonuniform exponential dichotomy of $\dot x=(A(t)-\gamma_1 I)x$. From the proof of $(a)$ we get that
$(-\infty,\gamma_1]\subset\rho(A)$. It is in contradiction with
the choice of $\gamma_1$. So we have $\dim\mathcal W_1\ge 1$.

For $i>1$, we have $\gamma_{i-1},\gamma_i\in\rho(A)$ and
$[\gamma_{i-1},\gamma_i]\cap\Sigma(A)\ne\emptyset$. By Proposition
\ref{po2.4} we have $\mathcal U_{\gamma_i}\supsetneq\mathcal
U_{\gamma_{i-1}}$. It follows that $\mathcal W_i=\mathcal
U_{\gamma_i}\cap\mathcal V_{\gamma_{i-1}}\supsetneq \mathcal
U_{\gamma_{i-1}}\cap\mathcal V_{\gamma_{i-1}}$ and consequently
$\dim\mathcal W_i\ge 1$.  This proves the claim.

Next we claim that $\mathcal V_{\gamma_i}=\mathcal
W_{i+1}+\mathcal V_{\gamma_{i+1}}$ for $i=0,1,\ldots,m-1$. In fact,
it follows from the fact that $\mathcal V_{\gamma_i}=\mathcal V_{\gamma_i}\cap(\mathcal
U_{\gamma_{i+1}}+\mathcal V_{\gamma_{i+1}})=\mathcal
V_{\gamma_i}\cap\mathcal U_{\gamma_{i+1}}+\mathcal
V_{\gamma_{i+1}}=\mathcal W_{i+1}+\mathcal V_{\gamma_{i+1}}$,
where in the second equality we have used the fact that $\mathcal
V_{\gamma_i}\supset\mathcal V_{\gamma_{i+1}}$.

Applying the last claim we have
\begin{eqnarray*}
\R\times \R^n&=&\mathcal U_{\gamma_0}+\mathcal
V_{\gamma_0}=\mathcal W_{0}+\mathcal V_{\gamma_0} =\mathcal
W_{0}+\mathcal W_1+\mathcal V_{\gamma_1}=\ldots
\\
&=& \mathcal W_{0}+\mathcal W_1+\ldots+\mathcal W_{m}+\mathcal
V_{\gamma_m}= \mathcal W_{0}+\mathcal W_1+\ldots+\mathcal
W_{m}+\mathcal W_{m+1}.
\end{eqnarray*}
Finally for $0\le i<j\le m+1$ we have $\mathcal W_i\cap\mathcal
W_j\subset\mathcal U_{\gamma_i}\cap\mathcal
V_{\gamma_{j-1}}\subset\mathcal U_{\gamma_i}\cap\mathcal
V_{\gamma_{i}}= \R\times\{0\}$. This proves that
$\R\times\R^n=\mathcal W_0\oplus\mathcal W_1\oplus
\ldots\oplus\mathcal W_{m+1}$ and consequently statement $(b)$.

We complete the proof of the theorem.\qed

\subsection{Proof of Theorem \ref{th1.2}}

By the assumption the evolution operator
$\Phi(t,s)$ of  system \eqref{1.1} has a nonuniformly bounded
growth, i.e. there exist $K\ge 1$, $a\ge 0$ and $\e\ge 0$ such
that
\begin{equation}\label{2.0}
\|\Phi(t,s)\|\le
Ke^{a|t-s|+\e|s|}, \quad t,s\in\R.
\end{equation}
First we claim that $\Sigma(A)\subset[-a-2\e,a+2\e]$, and so it is bounded.

For $\gamma>a+2\e$, we get from \eqref{2.0} that
\[
\|\Phi_\gamma(t,s)\|\le Ke^{(-\gamma+a)(t-s)+\e|s|},\quad \mbox{
for } t\ge s.
\]
Since $-\gamma+a+\e<-\e\le 0$, system $\dot x=(A(t)-\gamma I)x$
admits a nonuniform exponential dichotomy with the invariant
projection $P(t)=I$. This shows that $\gamma\in\rho(A)$ and
consequently $(a+2\e,\infty)\subset\rho(A)$.

For $\gamma<-a-2\e$, we have
\[
\|\Phi_\gamma(t,s)\|\le Ke^{-(\gamma+a)(t-s)+\e|s|},\quad \mbox{
for } t\le s.
\]
Since $-\gamma-a-\e>\e\ge 0$, system $\dot x=(A(t)-\gamma I)x$
admits a nonuniform exponential dichotomy with the invariant
projection $P(t)=0$. Hence we have
$(-\infty,-a-2\e)\subset\rho(A)$. Consequently
$\Sigma(A)\subset [-a-2\e,a+2\e]$. The claim follows.

Next we prove that $\Sigma(A)\ne \emptyset$. The above proof
implies that for $\gamma>a+2\e$, $\mathcal U_\gamma={\rm
Im}P=\R\times\R^n$ and $\mathcal V_\gamma={\rm
Ker}P=\R\times\{0\}$ because $P(t)=I$, and that for
$\gamma<-a-2\e$, $\mathcal U_\gamma={\rm Im}P=\R\times\{0\}$ and
$\mathcal V_\gamma={\rm Ker}P=\R\times\R^n$ because $P(t)=0$. Set
\[
\gamma^*=\sup\{\gamma\in\rho(A);\,\mathcal
V_\gamma=\R\times\R^n\}.
\]
Then $\gamma^*\in[-a-2\e,a+2\e]$. Moreover we have
$\gamma^*\in\Sigma(A)$. Otherwise, by Proposition \ref{po2.3} there
exists a neighborhood $J$ of $\gamma^*$ such that $J\subset
\rho(A)$ and for any $\gamma\in J$ we have $\mathcal
V_\gamma=\mathcal V_{\gamma^*}$. This is in contradiction with the
definition of $\gamma^*$. So $\Sigma(A)\ne\emptyset$. This proves
statement $(a)$.

Let $\Sigma(A)=[a_1,b_1]\cup\ldots\cup
[a_m,b_m]\subset[-a-2\e,a+2\e]$ with $m\ge 1$ and $-\infty<a_1\le
b_1<\ldots<a_m\le b_m<\infty$.  Then statement $(b)$ follows from Theorem
\ref{th1.1}, i.e.
\[
\mathcal W_0\oplus\mathcal W_1\oplus\ldots\oplus\mathcal W_m\oplus\mathcal W_{m+1}=\R\times\R^n.
\]
We complete the proof of the theorem. \qed

\section{Proof of Theorem \ref{th1.3}}\label{s3}

\setcounter{section}{3}
\setcounter{equation}{0}\setcounter{theorem}{0}

For proving Theorem \ref{th1.3} we need some preliminary results, which will be
presented in the next subsection.

\subsection{Preparation to the proof of Theorem \ref{th1.3}}

\begin{lemma}\label{le3.1}
The following statements are equivalent.
\begin{itemize}
\item[$(a)$] Systems \eqref{1.1} and \eqref{1.4} are nonuniformly kinematically similar via a transformation $x=S(t)y$.

\item[$(b)$] $\Phi_A(t,s)S(s)=S(t)\Phi_B(t,s)$ for all $t,s\in\R$,
where $\Phi_A$ and $\Phi_B$ are the evolution operators of systems
\eqref{1.1} and \eqref{1.4}, respectively.

\item[$(c)$] $S(t)$ is a solution of $\dot S=A(t)S-SB(t)$.
\end{itemize}
\end{lemma}

\begin{proof}
See Lemma 2.1 of \cite{Si02JLMS}, \cite{Co78} and \cite{Sie2014}.
\end{proof}

\begin{lemma}\label{le3.2}
 If systems \eqref{1.1} and \eqref{1.4} are nonuniformly kinematically similar, then they have the same nonuniform dichotomy spectrum.
\end{lemma}

\begin{proof}
It follows from statement $(b)$ of Lemma \ref{le3.1} and the fact
that $S(t)$ is a nonuniformly Lyapunov matrix. For more details, see e.g. Lemma 3.6 of \cite{Sie2014}.
\end{proof}

\begin{lemma}\label{le3.3}
Let $P_0\in\R^{n\times n}$ be a symmetric projection and
$\X(t)\in{\rm GL}_n(\R)$ the group of invertible matrix functions
in $t\in \R$. Set
$Q(t)=P_0X(t)^TX(t)P_0+(I-P_0)X(t)^TX(t)(I-P_0)$. Then
\begin{itemize}
\item[$(a)$] $Q(t)$ is positively definite and symmetric.

\item[$(b)$] There exists a unique positively definite and
symmetric matrix function $R(t)$ such that $R(t)^2=Q(t)$ and
$P_0R(t)=R(t)P_0$.

\item[$(c)$] $S(t)=X(t)R(t)^{-1}$ is invertible and satisfies
$S(t)P_0S(t)^{-1}=X(t)P_0X(t)^{-1}$ and
\[
\|S(t)\|\le \sqrt{2},\quad
\|S(t)^{-1}\|\le\sqrt{\|X(t)P_0X(t)^{-1}\|^2+\|X(t)(I-P_0)X(t)^{-1}\|^2}
\]
\end{itemize}
\end{lemma}

\begin{proof} See Lemma A.5 of \cite{Si02JLMS} and Lemma 3.2 of \cite{Sie2014}.
\end{proof}

\begin{lemma}\label{le3.4}
Assume that system \eqref{1.1} has an invariant projection
$P:\,\R\rightarrow\R^{n\times n}$ with $P(t)\ne 0,I$. Then there
exists a differentiable nonuniform Lyapunov matrix function $S: \, \R\rightarrow{\rm
GL}_n(\R)$ such that
\[
S(t)^{-1}P(t)S(t)=\left(\begin{array}{cc} I & 0\\ 0 &
0\end{array}\right) \quad \mbox{ for all } t\in \R.
\]
\end{lemma}

\begin{proof}
Since $P(t)$ is an invariant projection associated with the
evolution operator $\Phi(t,s)$ of system \eqref{1.1}, i.e.
$P(t)\Phi(t,s)=\Phi(t,s)P(s)$ for $t,s\in \R$, it forces that
$P(t)$ and $P(s)$ for any $t,s\in\R$ are similar and so have the
same rank. The fact that $P(t)$ is a projection implies that for
any given $s\in \R$ there exists a $T(s)\in{GL}_n(\R)$ such that
\begin{equation}\label{3.00}
T(s)P(s)T(s)^{-1}=\left(\begin{array}{cc}I_{n_1\times n_1} &
0_{n_1\times
n_2}\\
0_{n_2\times n_1} & 0_{n_2\times n_2}\end{array}\right)=P_0 \quad
\mbox{ for all }s\in \R,
\end{equation} where $n_1=\dim{\rm Im}P$
and $n_2=\dim{\rm Ker}P$. Applying Lemma \ref{le3.3} to
$X(t)=\Phi(t,s)T(s)^{-1}$ and $P_0$ we get a $R(t)$ satisfying
$P_0R(t)=R(t)P_0$ for $t\in \R$. Set
$S(t)=\Phi(t,s)T(s)^{-1}R(t)^{-1}$, we have
\[
S(t)^{-1}P(t)S(t)=R(t)T(s)P(s)T(s)^{-1}R(t)^{-1}=R(t)P_0R(t)^{-1}=P_0,
\]
where we have used the fact that $\Phi(t,s)\Phi(s,t)=I$ and the
invariance of $P(t)$ with respect to $\Phi(t,s)$.

Finally, the fact that $S(t)$ is a nonuniform Lyapunov matrix function follows
from the expression of $P_0$ and statement $(c)$ of Lemma \ref{le3.3}. For more details, see the proof of Theorem 3.8 of \cite{Sie2014}.
We complete the proof of the lemma.
\end{proof}

\subsection{Proof of Theorem \ref{th1.3}}

By the assumptions and Theorem \ref{th1.2}, we have $m\ge 1$,
$a_1>-\infty$, $a_m<\infty$ and $\mathcal W_0\oplus\mathcal W_1\oplus\ldots\oplus\mathcal W_m\oplus \mathcal W_{m+1}=\R\times \R^n$. Moreover it follows from
Theorem \ref{th1.1} that $\dim\mathcal W_i\ge 1$ for
$i=1,\ldots,m$.

In what follows we call the open intervals $(b_0,a_1),(b_1,a_2),\ldots, (b_{m-1},a_m)$ and
$(b_m,a_{m+1})$ the {\it spectral gaps}, where
$b_0=-\infty$ and $a_{m+1}=\infty$. Choose
$\gamma_i\in(b_i,a_{i+1})$ for $i=0,1,\ldots,m$. By Theorem \ref{th1.1} we have $\mathcal W_0=\mathcal U_{\gamma_0}$ and $\mathcal W_{m+1}=\mathcal V_{\gamma_m}$. The following proof combines the origin version of this paper and that of Theorem 3.9 of \cite{Sie2014}.

For any given $\gamma_0\in(-\infty,a_1)$, since $(-\infty, \gamma_0]\subset \rho(A)$, the system
\[
\dot x=\left(A(t)-\gamma_0 I\right) x,
\]
admits a nonuniform exponential dichotomy with an invariant projection $\widetilde P_0$. Then we have
\begin{equation}\label{3.1}
\begin{array}{l}
\qquad\,\,\,\|\Phi(t,s)\widetilde P_0(s)\|\le K_1e^{(\gamma_0+\alpha_1)(t-s)+\mu_1|s|}
\quad \mbox{
for } t\ge s,\\
\|\Phi(t,s)(I-\widetilde P_0(s))\|\le K_1e^{(\gamma_0+\beta_1)(t-s)+\nu_1|s|} \quad
\mbox{ for } t\le s,
\end{array}
\end{equation}
where $K_1\ge 1$, $\alpha_1<0$, $\beta_1>0$, $\mu_1,\nu_1\ge 0$,
and $\alpha_1+\mu_1<0$ and $\beta_1-\nu_1>0$.

We claim that system \eqref{1.1} is nonuniformly kinematically similar to
\begin{equation}\label{3.2}
\dot y=\left(\begin{array}{cc} B_0(t) & 0\\
0 & B_{11}(t)\end{array} \right)y
\end{equation}
with $B_0:\,\R\rightarrow\R^{n_0\times n_0}$ and $B_{11}:\,\R\rightarrow\R^{m_1\times m_1}$ differentiable, where $n_0=\dim{\rm Im}P_0$ and $m_1=\dim{\rm Ker}P_0$. Moreover $\Sigma(B_0)=\emptyset$ and $\Sigma(B_{11})=\Sigma(A)$.

Now we prove this claim. By Lemma \ref{le3.4} there exists a
differentiable nonuniform Lyapunov matrix function $S_0:\,\R\rightarrow{\rm GL}_n(\R)$
such that
\[
S_0(t)^{-1}\widetilde P_0(t)S_0(t)=\left(\begin{array}{cc}I_{n_1\times
n_1} & 0_{n_1\times n_2}\\
0_{n_2\times n_1} & 0_{n_2\times n_2}\end{array}\right)= P_0.
\]
Define
\[
B(t)=S_0(t)^{-1}(A(t)S_0(t)-\dot S_0(t))\quad \mbox {for } t\in\R.
\]
Lemma \ref{le3.1} means that system \eqref{1.1} is nonuniformly kinematically similar to
\begin{equation}\label{3.3}
\dot y=B(t)y,
\end{equation}
via the transformation $x(t)=S_0(t)y(t)$, and that
$S_0(t)^{-1}\Phi(t,s)S_0(s)$ is a fundamental matrix solution of
\eqref{3.3}.

Set $R(t)=S_0(t)^{-1}\Phi(t,s)T^{-1}$. From the proof of Lemma
\ref{le3.4} we have $P_0R(t)=R(t)P_0$. This implies that
$R(t)^{-1}$ and $\dot R(t)$ both commute with $P_0$. Using the
fact that $S_0(t)^{-1}\Phi(t,s)S_0(s)$ is a fundamental matrix
solution of \eqref{3.3}, we can prove easily that
\begin{equation}\label{3.4}
B(t)=\dot R(t)R(t)^{-1}\quad \mbox{and} \quad P_0B(t)=B(t)P_0.
\end{equation}
Write $B(t)$ in the block form, i.e.
\[
B(t)=\left(\begin{array}{cc}B_0 & C_0\\
C_{11} & B_{11}\end{array}\right),
\]
where $B_0:\R\rightarrow\R^{n_0\times n_0}$,
$B_{11}:\R\rightarrow\R^{m_1\times m_1}$,
$C_0:\R\rightarrow\R^{n_0\times m_1}$ and
$C_{11}:\R\rightarrow\R^{m_1\times n_0}$. From the expression of
$P_0$ and the second equation of \eqref{3.4} we get that
$C_0(t)=0$ and $C_{11}(t)=0$.

By Lemma \ref{le3.2} systems \eqref{1.1} and \eqref{3.2} have the
same nonuniform dichotomy spectrum. Moreover the evolution
operator of system \eqref{3.2} has the invariant projection $P_0$
given in \eqref{3.00}. So we get from \eqref{3.00} and \eqref{3.1} that $\Sigma(B_0)\subset (-\infty, a_1)$. This implies that $\Sigma(B_0)=\emptyset$.
The claim follows.

 For $\lambda\in (b_1,a_2)$, system $\dot
y=(B(t)-\lambda I)y$ admits a nonuniform exponential dichotomy with an invariant projection $\widetilde P_1$. So
there exist $\widetilde K_1,\widetilde \alpha<0,\widetilde
\beta>0,\widetilde \mu,\widetilde \nu\ge 0$ with $\widetilde
\alpha+\widetilde\mu<0$ and $\widetilde\beta-\widetilde\nu>0$ such
that
\begin{equation}\label{3.6}
\begin{array}{l}
\|\Phi_1(t,s)\widetilde P_1\|\le \widetilde
K_1e^{(\gamma+\widetilde\alpha_1)(t-s)+\widetilde\mu_1|s|} \quad
\mbox{
for } t\ge s,\\
\|\Phi_1(t,s)(I-\widetilde P_1\|\le \widetilde
K_1e^{(\gamma+\widetilde\beta_1)(t-s)+\widetilde\nu_1|s|} \quad
\mbox{ for } t\le s,
\end{array}
\end{equation}
where $\Phi_1(t,s)$ is the evolution operator of system $\dot y=B(t)y$.

Since $\Sigma(B_{11})=\Sigma(A)$, it follows from the last claim that system
\[
\dot y_1=B_{11}(t)y_1,
\]
is nonuniformly kinematically similar to
\[
\dot z_1=\left(\begin{array}{cc}B_{1}(t) & 0\\
0 & B_{22}(t)\end{array}\right)z_1,
\]
via a  nonuniformly Lyapunov transformation $y_1=S_{11}(t)z_1$. Take
\[
S_1(t)=\left(\begin{array}{cc}I_{n_1\times n_1} & 0\\
0 & S_{11}(t)\end{array}\right)S_0(t).
\]
Then system \eqref{1.1} is nonuniformly kinematically similar to
\begin{equation}\label{3.7}
\dot z=\widetilde B_1(t) z,\qquad \widetilde B_1(t)=\left(\begin{array}{ccc}B_{0}(t) & 0 & 0\\
0 & B_{1}(t)  & 0\\
0& 0& B_{22}(t)
\end{array}\right),
\end{equation}
via the nonuniformly Lyapunov transformation $x(t)=S_1(t)z(t)$. Since the
first inequality of \eqref{3.6} also holds for all $\gamma\ge a_2$, taking into account equation \eqref{3.7} we get that $\Sigma(\mbox{diag}(B_0,B_{1}))\subset (-\infty,a_2)$. Similarly from
the second inequality of \eqref{3.6} we have
$\Sigma(B_{22})\subset(b_1,\infty)$. Hence we have $\Sigma(B_{1})=[a_1,b_1]$.

According to the above process, we get a nonuniform Lyapunov transformation $x(t)=\widetilde S(t)w(t)$, which send system \eqref{1.1} to
\begin{equation}\label{3.8}
\dot w=\widetilde B_{m-1}(t) w, \qquad \widetilde B_{m-1}(t)=\left(\begin{array}{ccccc}
B_0(t) &          &        &             & \\
       & B_{1}(t)   &        &            &   \\
       &          & \ddots &             & \\
       &          &        &   B_{m-1}(t)  &   \\
          &          &        &    &   B_{mm}(t)
  \end{array}
 \right),
\end{equation}
with $\Sigma(B_0)=\emptyset$, and $\Sigma(B_{i})=[a_i,b_i]$ for $i=1,\ldots,m-1$. Take $\gamma_m\in( b_m,\infty)$, system $\dot w=(\widetilde B_{m-1}(t)-\gamma_m I)w$ admits a nonuniform exponential dichotomy with an invariant projection $\widetilde P_{m}$. Using the same arguments as in the above proof, system \eqref{1.1} is nonuniformly kinematically similar to system \eqref{1.5}. Again as in the above proof we get that $\Sigma(\mbox{diag}(B_0,\ldots,B_m))\subset (-\infty, \gamma_m)$ and $\Sigma(B_{m+1})\subset (\gamma_m,\infty)$. This implies that $\Sigma(B_m)=[a_m,b_m]$ and $\Sigma(B_{m+1})=\emptyset$.

Finally, we prove that the order $n_i$ of the matrix $B_i(t)$ in \eqref{1.5} is equal to $\dim \mathcal W_i$. Since  $\mathcal W_0=\mathcal U_{\gamma_0}$ for $\gamma_0\in (-\infty,b_1)$, by Proposition \ref{po2.2}  it follows that $\dim \mbox{Im} \widetilde P_0=\dim \mathcal U_{\gamma_0}=\dim \mathcal W_0$. In addition, the order $n_0$ of $B_0(t)$ is equal to $\dim\mbox{Im}\widetilde P_0$. These verify that $n_0=\dim \mathcal W_0$. Note that $\gamma_0\in (-\infty,a_1)$ and $\gamma_1\in (b_1,a_2)$, we get from Propositions \ref{po2.2} and \ref{po2.4} that
\[
\dim \mbox{Im} \widetilde P_1=\dim \mathcal U_{\gamma_1}=\dim(\mathcal U_{\gamma_1}\cap(\mathcal U_{\gamma_0}\oplus\mathcal V_{\gamma_0}))=\dim\mathcal W_0+\dim\mathcal W_1,
\]
where we have used the facts that $\mathcal U_{\gamma_0}\subset \mathcal U_{\gamma_1}$ and $\mathcal W_1=\mathcal U_{\gamma_1}\cap\mathcal V_{\gamma_0}$.
This implies that $n_1=\dim\mathcal W_1$ because $n_0+n_1=\dim\mbox{Im}\widetilde P_1$. Similarly $n_2=\dim\mathcal W_2$ follows from the facts that
\[
n_0+n_1+n_2=\dim \mbox{Im} \widetilde P_2=\dim \mathcal U_{\gamma_2}=\dim(\mathcal U_{\gamma_2}\cap(\mathcal U_{\gamma_1}\oplus\mathcal V_{\gamma_1}))=\dim\mathcal U_1+\dim\mathcal W_2,
\]
where $\gamma_2\in (b_2,a_2)$. By induction we can prove that $n_i=\dim\mathcal W_i$ for $i=1,\ldots,m$. For $\gamma_m\in (b_m,\infty)$ we get from Proposition  \ref{po2.2} again that
\[
 n_{m+1}=\dim \mbox{Ker} \widetilde P_m=\dim \mathcal V_{\gamma_m}= \dim\mathcal W_{m+1}.
\]
We complete the proof of the theorem. \qed

\section{Proof of Theorem \ref{th1.4}}\label{s4}

\setcounter{section}{4}
\setcounter{equation}{0}\setcounter{theorem}{0}

To simplify the proof, in the next subsection we first introduce
some basic knowledge on the tensor product and then present some
necessary preliminary results on the linear operators defined in the space of the vector--valued
homogeneous polynomials.

\subsection{The tensor product and its applications}

Let $V_i$ for $i=1,\ldots,k$ be $n_i$ dimensional real vector
spaces and let $V=V_1\otimes\ldots \otimes V_k$ be their tensor
product. Then $V$ is an $n_1\ldots n_k$ dimensional real vector
space. The following properties on the tensor product can be found
in Lemma 5.4.1 of \cite{Ar98}, which will be used later on.

\begin{proposition}\label{po4.1}
On the tensor product the following statements hold.

\begin{itemize}

\item[$(i)$] The splitting $V_1=\oplus_{i=1}^{p_1}V_i^{(1)}$ and
$V_2=\oplus_{j=1}^{p_2}V_i^{(2)}$ induce the splitting
\[
V_1\otimes V_2=\oplus_{(i,j)=(1,1)}^{(p_1,p_2)} V_i^{(1)}\otimes
V_j^{(2)}.
\]

\item[$(ii)$] Let $T_i: V_i\rightarrow V_i$ for $i=1,\ldots,4$ be
linear operators defined in the vector spaces $V_i$ of dimension $n_i$. Then
\[
T_1\otimes T_2(x\otimes y)=T_1 x\otimes T_2 y\qquad \mbox{ for }
x\in V_1,\,\, y\in V_2,
\]
defines a linear operator $T_1\otimes T_2$ in $V_1\otimes V_2$.
Moreover we have
\begin{eqnarray*}
(T_1+T_2)\otimes T_3&=&T_1\otimes T_3+T_2\otimes T_3,\\
T_1\otimes(T_2+T_3)&=& T_1\otimes T_2+T_1\otimes T_3,\\
 (T_1\otimes
T_2)\circ (T_3\otimes T_4)&=&(T_1\circ T_3)\otimes
(T_2\circ T_4),\\
(T_1\otimes T_2)^{-1}&=&T_1^{-1}\otimes T_2^{-1}\quad \mbox{ if }
T_1 \mbox{ and } T_1 \mbox{ are invertible},\\
\det(T_1\times T_2)&=&(\det T_1)^{n_1}(\det T_2)^{n_2},\quad
\|T_1\otimes T_2\|=\|T_1\|\|T_2\|,
\end{eqnarray*}
where $T_i+T_j$ makes sense if they are defined in the same vector
space.

\item[$(iii)$] If $A=(a_{ij})$ and $B=(b_{ij})$ are the matrix
representations of $T_1$ and $T_2$ respectively, then $T_1\otimes
T_2$ has the matrix representation
\[
A\otimes B=\left(\begin{array}{ccc} a_{11} B & \cdots & a_{1,n_1}
B\\ \vdots &\ddots & \vdots\\
a_{n_1,1}B &\cdots & a_{n_1,n_1}B\end{array}\right),
\]
which is called the Kronecker product of $A$ and $B$.
\end{itemize}
\end{proposition}

\begin{proposition}\label{po4.2}
If $\Phi_1(t,s)$ and $\Phi_2(t,s)$ are the evolution operators of
$\dot z_1=A(t)z_1$ and $\dot z_2=B(t)z_2$ respectively, then
$\Phi_1(t,s)\otimes \Phi_2(t,s)$ is the evolution operator of
\[
\dot z=(A_1(t)\otimes I_2+I_1\otimes A_2(t))z.
\]
\end{proposition}

Now we recall some results related to the linear operators defined
in the space of the vector--valued homogeneous polynomials, part of
them can be found in Chapter 8 of \cite{Ar98}.

Let
\[
H_{n,k}(\R^d)=\left\{f=\sum\limits_{\tau\in\Z_+^n} f_\tau
x^\tau;\, f_\tau\in\R^d,|\tau|=k\right\},
\]
be the vector space of homogeneous polynomials in $n$ variables of
degree $k$ with their values in $\R^d$. Then
\[
H_{n,k}(\R^d)=H_{n,k}(\R^1)\otimes \R^d.
\]
Set $D=\dim (H_{n,k}(\R^1))$, we have
$
D=\left(\begin{array}{c} n+k-1\\
k\end{array}\right),
$
and $H_{n,k}(\R^1)$ and $\R^D$ are equivalent, denoted by $ H_{n,k}(\R^1)\cong \R^D$.
Let $(u_1,\ldots,u_d)$ be a basis of $\R^d$, and let
$\{x^\tau;\,\tau\in\Z_+^n,|\tau|=k\}$ be a basis of
$H_{n,k}(\R^1)$. Clearly $\{x^\tau u_i;\,
i=1,\ldots,d,\tau\in\Z_+^n,|\tau|=k\}$ is a basis of
$H_{n,k}(\R^d)$. Under this basis $H_{n,k}(\R^d)\cong \R^{Dd}$ via
the equivalence
\[
f=\sum\limits_{i=1}\limits^{d}\sum\limits_{|\tau|=k}f_{\tau,i}(t)x^\tau
u_i\longrightarrow (f_{\tau,i})\in\R^{Dd}.
\]

For any $n\times n$ matrix $A(t)$ we define a $D\times D$ matrix
\[
N(A)_k=\left(N_{\tau \sigma}^{(k)}(A)\right),\quad
(Ax)^\tau=\sum\limits_{\sigma\in\Z_+^n,|\sigma|=k}N_{\sigma
\tau}^{(k)} (A) x^\sigma,\quad \tau\in\Z_+^n, |\tau|=k .
\]
Usually the entries of $N(A)_k$ are nonlinear functions of the
entries of $A$.

\begin{proposition} \label{po4.3} Let $A,B$ be $n\times n$
matrices, and $k\ge 2$. The following statements hold.
\begin{itemize}

\item[$(i)$] $\|N(A)_k\|\le c\|A\|^k$, $N(I_2)=I_1$ and
$N(AB)_k=N(B)_kN(A)_k$, where $c$ is independent of $A$, and $I_2$ and $I_1$ are respectively
the $n\times n$ and  $D\times D$ unit matrices.

\item[$(ii)$] If $A$ is invertible then
$N(A^{-1})_k=(N(A)_k)^{-1}$.

\item[$(iii)$] If $A(t)=\mbox{\rm diag}(A_1(t),\ldots,A_p(t))$
with $A_i(t):\R\rightarrow \R^{n_i\times n_i}$ for $i=1,\ldots,p$,
then there exists a $D\times D$ permutation matrix $P$ independent
of $t$ under which $N(A)_k$ is similar to a block diagonal matrix
\[
\mbox{\rm diag}(N(A)_\tau,\,\tau\in\Z_+^p,|\tau|=k),
\]
with $N(A)_\tau:\R\rightarrow \R^{q_\tau\times q_\tau}$ and $
q_\tau=\prod\limits_{i=1}\limits^p\left(\begin{array}{c}\tau_i+n_i-1\\
\tau_i\end{array}\right)$. Moreover we have
\[
\|N(A)_\tau(t)\|\le c\prod\limits_{i=1}\limits^p\|A_i\|^{\tau_i},
\]
where $c$ is independent of $A(t)$.
\end{itemize}
\end{proposition}

\begin{proof}
Statements $(i)$ and $(ii)$ can be found in Lemma 8.1.2 of
\cite{Ar98}. The proof of statement $(iii)$ is given in
Proposition 5 of \cite{WL08}.
\end{proof}

For $k\ge 2$ we define a linear operator
\begin{eqnarray*}
T(A)_k:\qquad H_{n,k}(\R^1)\quad &\longrightarrow & \quad H_{n,k}(\R^1)\nonumber\\
 h=\sum\limits_{\tau\in\Z_+^n,|\tau|=k}h_\tau x^\tau &\longrightarrow& \frac{\p
h}{\p x} Ax\triangleq T(A)_k(h).
\end{eqnarray*}
Proposition 8.3.4 of \cite{Ar98} and Proposition 6 of \cite{WL08}
established a relation between the evolution operators
$\Phi_{-T(A)_k}(t)$ and $\Phi_A(t)$.
\begin{proposition}\label{po4.4}
Let $\Phi_A(t,s)$ be the evolution operator of $\dot x=A(t)x$.
Then
\[
\Phi_{-T(A)_k}(t,s)=N(\Phi_{-A}(t,s))_k=N(\Phi_A(t,s))_k^{-1}.
\]
\end{proposition}

Now we define the  linear operator
\begin{eqnarray}\label{ell}
L_k(t):\qquad H_{n,k}(\R^n) &\longrightarrow & \quad H_{n,k}(\R^n)\\
 h(x) &\longrightarrow& A(t)h(x)-\frac{\p
h(x)}{\p x} A(t)x.\nonumber
\end{eqnarray}
Since $H_{n,k}(\R^n)=H_{n,k}(\R^1)\otimes \R^n$, we can write the
operator $L_k(t)$ in
\[
L_k(t)=I_1\otimes A-T(A)_k\otimes I_2,
\]
where $I_1$ and $I_2$ are the unit matrices on $H_{n,k}(\R^1)$ and
$\R^n$, respectively.

By Propositions \ref{po4.3} and \ref{po4.4}, we get from
Proposition 8.3.4 of \cite{Ar98} and Proposition 6 of \cite{WL08}
a relation between the evolution operators $\Phi_{L_k}(t)$ and
$\Phi_A(t)$.
\begin{proposition}\label{po4.5}
Let $\Phi_A(t,s)$ be the evolution operator of $\dot x=A(t)x$.
Then the following statements hold.
\begin{itemize}

\item[$(a)$] $
\Phi_{L_k}(t,s)=\Phi_{-T(A)_k}(t,s)\otimes\Phi_A(t,s)
=N(\Phi_A(t,s))_k^{-1}\otimes\Phi_A(t,s)=N(\Phi_A(s,t))\otimes\Phi_A(t,s)$.

\item[$(b)$] If $A=\mbox{\rm diag}(A_1(t),\ldots,A_p(t))$, then
\begin{eqnarray*}
N(\Phi_A(s,t))\otimes\Phi_A(t,s)&=&\mbox{\rm
diag}(N(\Phi_A(s,t))_\tau\otimes\Phi_{A_j}(t,s),\,\tau\in\Z_+^p,|\tau|=k,j=1,\ldots,p),\\
\|N(\Phi_A(s,t))_\tau\|&\le &
c\prod\limits_{i=1}\limits^{p}\|\Phi_{A_i}(s,t)\|^{\tau_i},
\end{eqnarray*}
where $c$ depends only on $n$, $k$ and the norm.
\end{itemize}
\end{proposition}

\subsection{Proof of Theorem \ref{th1.4}}

To simplify notations we write system \eqref{1.7} in
\begin{equation}\label{4.21}
\dot x=A(t)x+f(t,x),
\end{equation}
with $A(t)=\mbox{\rm diag}(A_1(t),\ldots,A_m(t))$ and
$f(t,x)=(f_1(t,x),\ldots,f_m(t,x))^T$, where $T$ denotes the
transpose of a matrix. Assume that there exists a near identity
formal transformation $x=y+h(t,y)$ under which system \eqref{4.21}
is transformed into
\begin{equation}\label{4.22}
\dot y=A(t)y+g(t,y),
\end{equation}
where $g(t,y)$ is a formal series in $y$. Then $h(t,y)$ should satisfy the
following equation
\begin{equation}\label{4.23}
\frac{\p h}{\p t}=A(t)h(t,y)-\frac{\p h}{\p
y}A(t)y+f(t,y+h(t,y))-\frac{\p h}{\p y}g(t,y)-g(t,y).
\end{equation}
Set $w(t,y)\sim \sum\limits_{i=2}\limits^\infty w_i(t,y)$ with
$w\in\{h,f,g\}$, and $w_i(t,y)$ a homogeneous polynomial of
degree $i$ in $y$. Equation \eqref{4.23} can be written in
\begin{equation}\label{4.23.1}
\frac{\p h_k(t,y)}{\p t}=L_k(t)h_k(t,y)+F_k(t,y)-g_k(t,y),\quad
k=2,3,\ldots
\end{equation}
where $F_k(t,y)$ is a homogeneous polynomial in $y$ of degree $k$
which is a function of $h_2,\ldots,h_{k-1}$ obtained from the
expansion of $f(t,y+h(t,y))-\frac{\p h}{\p y}g(t,y)$.
We note that $F_k(t,y)$ are successively known.
Recall that $L_k(t)$ is the linear operator defined in \eqref{ell}.

In the base $\{x^\tau u_i;\,\tau\in\Z_+^p,|\tau|=k,i=1,\ldots,n\}$
of $H_{n,k}(\R^n)$ each homogeneous polynomial $w_k(t,y)$ with
$w\in\{h,F,g\}$ is uniquely determined by its coefficients. Let
$w_k(t)$  with $w\in\{h,F,g\}$ be a vector--valued function of
dimension $Dn$ which is formed by the coefficients of $w_k(t,y)$
in the given base. Then we get from \eqref{4.23.1} that
\begin{equation}\label{4.24}
\frac{d}{dt}h_k(t)=L_k(t)h_k(t)+F_k(t)-g_k(t),
\end{equation}
where for simplicity to notations we still use $L_k(t)$ to denote the linear operator acting on $h_k(t)$.

Since $A(t)$ is a block diagonal matrix, by Proposition \ref{po4.5} the
evolution operator $\Phi_{L_k}(t,s)$ of
$\frac{d}{dt}h_k(t)=L_k(t)h_k(t)$ is also a block diagonal matrix.
According to the block diagonal form of $\Phi_{L_k}(t,s)$ given in
Proposition \ref{po4.5} we separate the vector space $\R^{Dn}$ in
the direct sum of the subspaces $\R^{q_\tau n_j}$ for
$j=1,\ldots,m,\tau\in\Z_+^m,|\tau|=k$, where $n_j$ is the order of
the matrix $A_j$ and $q_\tau$ is defined in statement $(iii)$ of
Proposition \ref{po4.3}. Correspondingly we have
\[
h_k(t)=\bigoplus\limits_{\tau,j}h_k^{(\tau,j)}(t).
\]
So system \eqref{4.24} can be written in
\begin{equation}\label{4.25}
\frac{d}{dt}h_k^{(\tau,j)}(t)=L_k^{(\tau,j)}(t)h_k^{(\tau,j)}(t)
+F_k^{(\tau,j)}(t)-g_k^{(\tau,j)}(t),
\end{equation}
with $\tau\in\Z_+^m,|\tau|=k$ and $j=1,\ldots,m$, where
$L_k^{(\tau,j)}(t)$ is the diagonal entry of the block diagonal
matrix $L_k(t)$.

Furthermore we separate
$p_k^{(\tau,j)}(t)=p_{k1}^{(\tau,j)}(t)+p_{k2}^{(\tau,j)}(t)$ with $p\in \{F,g\}$ in
such a way that the former is corresponding to those $(\tau,j)$
such that
$[a_j,b_j]\cap\sum\limits_{i=1}\limits^m\tau_i[a_i,b_i]=\emptyset$
and the latter is corresponding to those $(\tau,j)$ such that
$[a_j,b_j]\cap\sum\limits_{i=1}\limits^m\tau_i[a_i,b_i]\ne\emptyset$.
According to this decomposition system \eqref{4.25} can be decomposed into two subsystems:
\begin{equation}\label{4.25.1}
\frac{d}{dt}h_k^{(\tau,j)}(t)=L_k^{(\tau,j)}(t)h_k^{(\tau,j)}(t)
+F_{k1}^{(\tau,j)}(t)-g_{k1}^{(\tau,j)}(t),
\end{equation}
\begin{equation}\label{4.25.2}
\frac{d}{dt}h_k^{(\tau,j)}(t)=L_k^{(\tau,j)}(t)h_k^{(\tau,j)}(t)
+F_{k2}^{(\tau,j)}(t)-g_{k2}^{(\tau,j)}(t).
\end{equation}

In the case
$[a_j,b_j]\cap\sum\limits_{i=1}\limits^m\tau_i[a_i,b_i]\ne\emptyset$,
i.e. for those $(\tau,j)$ the nonuniform dichotomy spectrum is
resonant, we choose $g_{k2}^{(\tau,j)}(t)=F_{k2}^{(\tau,j)}(t)$, and
consequently equation \eqref{4.25.2} has the trivial solution
$h_{k}^{(\tau,j)}(t)=h_{k2}^{(\tau,j)}(t)=0$.

In the case
$[a_j,b_j]\cap\sum\limits_{i=1}\limits^m\tau_i[a_i,b_i]=\emptyset$,
i.e. for those $(\tau,j)$ the nonuniform dichotomy spectrum is not
resonant,  we have either $a_j>\tau_1b_1+\ldots+\tau_mb_m$ or
$b_j<\tau_1a_1+\ldots+\tau_ma_m$. In this case for any value of
$g_{k1}^{(\tau,j)}(t)$ equation \eqref{4.25.1} has a unique solution.
For simplicity we take $g_{k1}^{(\tau,j)}(t)=0$. By the variation of
constants formula we obtain from \eqref{4.25.1} that
$h_k^{(\tau,j)}(t)=h_{k1}^{(\tau,j)}(t)$ with either
\begin{equation}\label{4.26}
h_{k1}^{(\tau,j)}(t)=\int\limits_{-\infty}\limits^t\Phi_{L_k}^{(\tau,j)}(t,s)
F_{k1}^{(\tau,j)}(s)ds, \quad \mbox{\rm if }
b_j<\tau_1a_1+\ldots+\tau_ma_m,
\end{equation}
or
\begin{equation}\label{4.27}
h_{k1}^{(\tau,j)}(t)=-\int\limits_{t}\limits^\infty\Phi_{L_k}^{(\tau,j)}(t,s)
F_{k1}^{(\tau,j)}(s)ds, \quad \mbox{\rm if }
a_j>\tau_1b_1+\ldots+\tau_mb_m,
\end{equation}
where $\Phi_{L_k}^{(\tau,j)}(t,s)$ is the evolution operator of
the linear equation
$\frac{d}{dt}h_k^{(\tau,j)}(t)=L_k^{(\tau,j)}(t)h_k^{(\tau,j)}(t)$.

Combining the two cases, we get $g_k^{(\tau,j)}=g_{k2}^{(\tau,j)}=F_{k2}^{(\tau,j)}$ and $h_k^{(\tau,j)}=h_{k1}^{(\tau,j)}$.
Since the integrals in \eqref{4.26} and \eqref{4.27} are improper, we need to prove that they are convergent and so the functions in \eqref{4.26} and \eqref{4.27} are well defined.

\begin{proposition} \label{p111}
Let $k\in (3,\,\sigma/\varrho)$, where $\sigma=\min\{-\alpha_i,\beta_i;\,i=1,\ldots,m\}$ and  $\varrho=\max\{\mu_i,\nu_i;\, i=1,\ldots,m\}$. The following statements hold.
\begin{itemize}
\item[$(a)$]  If  $b_j<\tau_1a_1+\ldots+\tau_ma_m$ we have for $2\le r<2k-4$
\begin{equation}\label{4.212}
\|h_{r1}^{(\tau,j)}(t)\|\le
C_{r,\tau,j}\exp\left(-\frac{r-1}{2}(2k-r-4)\varrho|t|\right),
\end{equation}
where $C_{r,\tau,j}$ is a positive constant.
\item[$(b)$] If $a_j>\tau_1b_1+\ldots+\tau_mb_m$, then $h_{r1}^{(\tau,j)}(t)$ satisfies the same estimation as that given in \eqref{4.212} for  $2\le r<2k-4$ with probably the coefficient $C_{r,\tau,j}$ different.
\end{itemize}
\end{proposition}

\begin{proof}

From Proposition \ref{po4.5} we have
\[
\Phi_{L_k}^{(\tau,j)}(t,s)=N(\Phi_A(s,t))_\tau\otimes
\Phi_{A_j}(t,s),
\]
and
\begin{equation}\label{4.29}
\|\Phi_{L_k}^{(\tau,j)}(t,s)\|=
\|N(\Phi_A(s,t))_\tau\|\|\Phi_{A_j}(t,s)\| \le
c_k\prod\limits_{i=1}\limits^{m}\|\Phi_{A_i}(s,t)\|^{\tau_i}\|\Phi_{A_j}(t,s)\|,
\end{equation}
where $c_k$ depends only on $k, n$ and the norm.

Under the nonresonant conditions we define
\[
D_{j\tau}=\left\{\begin{array}{l}
\tau_1a_1+\ldots+\tau_ma_m-b_j\quad
\mbox{\rm if } b_j<\tau_1a_1+\ldots+\tau_ma_m,\\
a_j-\tau_1b_1-\ldots-\tau_mb_m \quad \mbox{\rm if }
a_j>\tau_1b_1+\ldots+\tau_mb_m.\end{array}\right.
\]
Set $\e_{j\tau}=D_{j\tau}/2(|\tau|+1)$. For
$\sigma_j\in[a_j-\e_{j\tau},a_j)$ and
$\rho_i\in(b_j,b_j+\e_{j\tau}]$, we can check that
\begin{eqnarray*}
\rho_j-\tau_1\sigma_1-\ldots-\tau_m\sigma_m&\le& -\frac 12
D_{j\tau}\quad\mbox{\rm if } b_j<\tau_1a_1+\ldots+\tau_ma_m,\\
\sigma_j-\tau_1\rho_1-\ldots-\tau_m\rho_m&\ge&\,\,\,\, \frac 12
D_{j\tau} \quad \mbox{\rm if } a_j>\tau_1b_1+\ldots+\tau_mb_m.
\end{eqnarray*}
So systems $\dot z=(A_i(t)-\sigma_i I)z$ and $\dot
z=(A_i(t)-\rho_i I)z$ admit nonuniform exponential dichotomies.
Hence there exist $K_i\ge 1$, $\alpha_i<0$, $\beta_i>0$, and
$\mu_i,\nu_i\ge 0$ with $\alpha_i+\mu_i<0$ and $\beta_i-\nu_i>0$
such that
\begin{equation}\label{4.210}
\begin{array}{l}
\|\Phi_{A_i}(t,s)\|\le
K_ie^{\rho_i(t-s)}e^{\alpha_i(t-s)+\mu_i|s|}\quad \mbox{\rm for } t\ge s,\\
\|\Phi_{A_i}(t,s)\|\le
K_ie^{\sigma_i(t-s)}e^{\beta_i(t-s)+\nu_i|s|}\quad \mbox{\rm for }
t\le s.
\end{array}
\end{equation}
Combining \eqref{4.29} and \eqref{4.210} we obtain that
\begin{equation}\label{4.211}
\|\Phi_{L_k}^{(\tau,j)}(t,s)\|\le \left\{\begin{array}{l}
K_{k,\tau,j}e^{\left[(\rho_j-\sum\limits_{i=1}\limits^m\tau_i\sigma_i)
+(\alpha_j-\sum\limits_{i=1}\limits^m\tau_i\beta_i)\right](t-s)
+\mu_j|s|+\sum\limits_{i=1}\limits^m\tau_i\nu_i|t|}\quad
\mbox{\rm for } t\ge s,\\
K_{k,\tau,j}e^{\left[(\sigma_j-\sum\limits_{i=1}\limits^m\tau_i\rho_i)
+(\beta_j-\sum\limits_{i=1}\limits^m
\tau_i\alpha_i)\right](t-s)+\nu_j|s|+\sum\limits_{i=1}\limits^m\tau_i\mu_i|t|}\quad
\mbox{\rm for } t\le s,\end{array}\right.
\end{equation}
where
$K_{k,\tau,j}=c_k\prod\limits_{i=1}\limits^{m}K_i^{\tau_i}K_j$. To
simplify the notation, for $b_j<\tau_1a_1+\ldots+\tau_ma_m$ we set
\[
\omega_{\tau,j}=\rho_j-\sum\limits_{i=1}\limits^m\tau_i\sigma_i+\alpha_j-\sum\limits_{i=1}\limits^m\tau_i\beta_i.
\]

\noindent{\it Statement $(a)$}. We prove this statement by induction. For $r=2$, by the assumptions of the theorem we have
\[
\|F_2^{(\tau,j)}(t)\|\le \|p_2(t)\|\le d_2
e^{-k\varrho|t|}.
\]
Recall that $p_2(t)$ is the coefficient vector of the vector--valued
homogeneous polynomial $\widetilde f_2(t,y)$ in the Taylor
expansion of $f(t,y)$. Then it follows from \eqref{4.26} and
\eqref{4.211} that for $b_j<\tau_1a_1+\ldots+\tau_ma_m$ we have
\[
\|h_{21}^{(\tau,j)}(t)\|\le
K_{2,\tau,j}d_2\int\limits_{-\infty}\limits^{t}e^{\omega_{\tau,j}(t-s)+\mu_j|s|+\sum\limits_{i=1}\limits^m\tau_i\nu_i|t|-k\varrho|s|}ds,
\]
where $|\tau|=2$.

If $t\le 0$, we have
\begin{eqnarray*}
\|h_{21}^{(\tau,j)}(t)\|&\le &
K_{2,\tau,j}d_2e^{\left(\omega_{\tau,j}-\sum\limits_{i=1}\limits^m\tau_i\nu_i\right) t}\int\limits_{-\infty}\limits^{t}e^{-(\omega_{\tau,j}+\mu_j -k\varrho)s}ds\\
&=&\frac{K_{2,\tau,j}d_2}{k\varrho-\omega_{\tau,j}-\mu_j}e^{\left(k\varrho-\sum\limits_{i=1}\limits^m\tau_i\nu_i-\mu_j\right)t}\\
&\le & \frac{K_{2,\tau,j}d_2}{k\varrho-\omega_{\tau,j}-\mu_j}e^{(k-3)\varrho t}
\le \frac{2K_{2,\tau,j}d_2}{D_{j\tau}}e^{(k-3)\varrho t},
\end{eqnarray*}
where we have used the facts that $|\tau|=2$, $\varrho=\max\{\mu_i,\nu_i;\, i=1,\ldots,m\}$ and  $k\varrho-\omega_{\tau,j}-\mu_j>-(\rho_j-\sum\limits_{i=1}\limits^m\tau_i\sigma_i)\ge\frac 12 D_{j\tau}$.

If $t>0$, we have
\begin{eqnarray*}
\|h_{21}^{(\tau,j)}(t)\|&\le &
K_{2,\tau,j}d_2e^{\left(\omega_{\tau,j}+\sum\limits_{i=1}\limits^m\tau_i\nu_i \right)t}
\left(\int\limits_{-\infty}\limits^{0}e^{-(\omega_{\tau,j}+\mu_j -k\varrho)s}ds+\int\limits_{0}\limits^{t}e^{-(\omega_{\tau,j}-\mu_j +k\varrho)s}ds\right)\\
&=&K_{2,\tau,j}d_2e^{\left(\omega_{\tau,j}+\sum\limits_{i=1}\limits^m\tau_i\nu_i \right)t}\times\\
&& \qquad
\left(\frac{1}{k\varrho-\omega_{\tau,j}-\mu_j}+
\frac{-1}{\omega_{\tau,j}-\mu_j +k\varrho}\left(e^{-(\omega_{\tau,j}-\mu_j +k\varrho)t}-1\right)\right)\\
&\le &
\frac{-K_{2,\tau,j}d_2}{\omega_{\tau,j}-\mu_j +k\varrho}e^{-\left(k\varrho-\mu_j-\sum\limits_{i=1}\limits^m\tau_i\nu_i\right)t}
\le \frac{2K_{2,\tau,j}d_2}{D_{j\tau}}e^{-(k-3)\varrho t}.
\end{eqnarray*}
We should mention that in the third inequality we have used the facts that $|\tau|=2$, $k\varrho+\alpha_j\le k\varrho-\sigma<0$ and
$\omega_{\tau,j}-\mu_j+k\varrho\le -\frac 12 D_{j\tau}<0$. In the second inequality we have used the fact  $1/(k\varrho-\omega_{\tau,j}-\mu_j)+1/(\omega_{\tau,j}-\mu_j +k\varrho)=2(k\varrho-\mu_j)/(k\varrho-\omega_{\tau,j}-\mu_j)(\omega_{\tau,j}-\mu_j +k\varrho)<0$, because $k\varrho-\mu_j>\varrho-\mu_j\ge 0$. This proves statement $(a)$ for $r=2$.

In order for using induction we assume that statement $(a)$ holds for $r<2k-5$. Consider the case $r+1$.
By the assumptions of the theorem and the construction of $F_{r+1}^{(\tau,j)}(t)$ there exists a constant $b_{r+1,\tau,j}$ such that
\begin{equation}\label{e111}
\|F_{r+1}^{(\tau,j)}(t)\|\le b_{r+1,\tau,j}
 e^{-\left(rk-\frac{(r-1)(r+4)}{2}\right)\varrho|t|}.
\end{equation}
In fact, $F_{l}^{(\tau,j)}(t)$ is the coefficient of the monomial $y^\tau$ in the $j$th component of the vector--valued homogeneous polynomial $F_{l}(t,y)$ in $y$ of degree $l$, and
\begin{equation}\label{e112}
F_l(t,y)=\left[\sum\limits_{r=2}\limits^l f_r\left(t,y+\sum\limits_{s=2}\limits^{l-1}h_s(t,y)\right)\right]_l-\sum\limits_{r=2}\limits^{l-1}\frac{\partial h_r(t,y)}{\partial y}g_{l+1-r}(t,y),
\end{equation}
where $[A(t,y)]_l$ denotes the homogeneous part of degree $l$ of a polynomial function $A(t,y)$ in $y$. The expression of $F_l(t,y)$ follows from the construction of the transformation $x=y+h(t,y)$ which sent system \eqref{4.21} to its normal form \eqref{4.22}.  Recall that $f_r$, $h_r$ and $g_r$ are the vector--valued homogeneous polynomials of degree $r$ in $y$ of the Taylor expansions of $f,h$ and $g$, respectively. Since $g_r(t,y)=F_r(t,y)=F_{r2}(t,y)$ and $h_r(t,y)=h_{r1}(t,y)$, so the estimation \eqref{e111} can be obtained from \eqref{e112} using the induction through  the estimations on the coefficients of $h_s, g_s$ for $s=2,\ldots,r$ and \eqref{eth141} (i.e. the estimation on the coefficients of $f_s$) for $s=2,\ldots,r+1$.

Now from \eqref{4.26}, \eqref{4.211} and
\eqref{e111} we get that
\[
\|h_{r+1,1}^{(\tau,j)}(t)\|\le
d_{r+1,\tau,j}\int\limits_{-\infty}\limits^{t}
e^{\omega_{\tau,j}(t-s)+\mu_j|s|+\sum\limits_{i=1}\limits^m\tau_i\nu_i|t|-\left(rk-\frac{(r-1)(r+4)}{2}\right)\varrho|s|}ds.
\]
where $|\tau|=r+1$ and $d_{r+1,\tau,j}=K_{r+1,\tau,j}b_{r+1,\tau,j}$.

If $t\le 0$, working in a similar way to the proof of the case $r=2$ and by direct integrating we get that
\begin{eqnarray*}
\|h_{r+1,1}^{(\tau,j)}(t)\|&\le&
\frac{d_{r+1,\tau,j}}{(r k-(r-1)(r+4)/2)\varrho-u_j-\omega_{\tau,j}}
e^{\left((rk-\frac{(r-1)(r+4)}{2})\varrho-\mu_j-\sum\limits_{i=1}\limits^n\tau_i\nu_i\right)t}\\
&\le & \frac{2d_{r+1,\tau,j}}{D_{\tau,j}}
e^{\left(rk-\frac{r(r+5)}{2} \right)\varrho t},
\end{eqnarray*}
where we have used the fact that $-(\mu_j+\sum\limits_{i=1}\limits^n\tau_i\nu_k)t\le -(r+2)t$.

If $t>0$, we have
\begin{eqnarray*}
\|h_{r+1,1}^{(\tau,j)}(t)\|&\le &
d_{r+1,\tau,j}e^{\left(\omega_{\tau,j}+\sum\limits_{i=1}\limits^m\tau_i\nu_i \right)t}
\left(\int\limits_{-\infty}\limits^{0}e^{-\left(\omega_{\tau,j}+\mu_j -(rk-\frac{(r-1)(r+4)}{2})\varrho\right)s}ds\right.\\
&&\qquad\qquad\qquad\qquad\qquad\quad \left.+\int\limits_{0}\limits^{t}e^{-\left(\omega_{\tau,j}-\mu_j +(rk-\frac{(r-1)(r+4)}{2})\varrho\right)s}ds\right)\\
&=&d_{r+1,\tau,j}e^{\left(\omega_{\tau,j}+\sum\limits_{i=1}\limits^m\tau_i\nu_i \right)t}
\left(\frac{1}{\left(rk-\frac{(r-1)(r+4)}{2}\right)\varrho-\omega_{\tau,j}-\mu_j}\right.\\
&& \quad \left.+
\frac{-1}{\omega_{\tau,j}-\mu_j +\left(rk-\frac{(r-1)(r+4)}{2}\right)\varrho}\left(e^{-\left(\omega_{\tau,j}-\mu_j +(rk-\frac{(r-1)(r+4)}{2})\varrho\right)t}-1\right)\right)\\
&\le &
\frac{-d_{r+1,\tau,j}}{\omega_{\tau,j}-\mu_j +\left(rk-\frac{(r-1)(r+4)}{2}\right)\varrho}e^{-\left((rk-\frac{(r-1)(r+4)}{2})\varrho-\mu_j-\sum\limits_{i=1}\limits^m\tau_i\nu_i\right)t}\\
&\le & \frac{2d_{r+1,\tau,j}}{D_{j\tau}}e^{-\left(rk-\frac{r(r+5)}{2}\right)\varrho t},
\end{eqnarray*}
where we have used the fact that $1/((rk-(r-1)(r+4)/{2})\varrho-\omega_{\tau,j}-\mu_j)+1/(\omega_{\tau,j}-\mu_j +(rk-(r-1)(r+4)/{2})\varrho)<0$, because $|\tau|=r+1>2$ and $\alpha_j-\sum\limits_{i=1}\limits^{m}\tau_i\beta_i+(rk-(r-1)(r+4)/2)\varrho<-(r+2)\sigma+rk\varrho<0$. Also we have used the fact  that $\omega_{\tau,j}-\mu_j+(rk-(r-1)(r+4)/2)\varrho\le-\frac 12 D_{j\tau}+\alpha_j-\sum\limits_{i=1}\limits^m\tau_i\beta_i-\mu_j+(rk-(r-1)(r+4)/2)\varrho<-\frac 12 D_{j\tau}$ because $\mu_j\ge 0$ and
$\alpha_j-\sum\limits_{i=1}\limits^m\tau_i\beta_i+(rk-(r-1)(r+4)/2)\varrho\le -(r+2)\sigma+(rk-(r-1)(r+4)/2)\varrho\le-k\varrho(r+2)+(rk-(r-1)(r+4)/2)\varrho=-2k\varrho-(r-1)(r+4)\varrho/2<0$. This proves statement $(a)$ for $r+1$. So by induction, statement $(a)$ follows.

\noindent{\it Statement $(b)$}. Its proof can be got from \eqref{4.211} and from the same arguments as those given in the proof of statement $(a)$. The details are omitted. We finish the proof of the proposition. \end{proof}

By Proposition \ref{p111}, the functions $h_{r1}^{(\tau,j)}(t)$ in \eqref{4.26} and \eqref{4.27} for $r=2,\ldots,2k-5$ are well defined and bounded.
Set $x=y+h(t,y)$ with
\[
h(t,y)=\sum\limits_{r=2}\limits^{2k-5}\left(\sum\limits_{|\tau|=r,1\le j\le m}\bigoplus h_{r1}^{(\tau,j)}y^\tau\right).
\]
Then $h(t,y)$ is a polynomial of degree $2k-5$ with the coefficients all bounded functions in $t\in \mathbb R$.
Hence, by the previous constructions we get that system \eqref{1.7} is transformed into \eqref{1.71} via the time dependent change of variables $x=y+h(t,y)$.

We complete the proof the theorem.\qed

\medskip

\noindent{\bf Acknowledgements.} The author sincerely appreciate the referee for his/her excellent comments, which improve the paper and correct some mistakes in the first version of this paper.

\end{document}